\documentclass{article}
\usepackage{amsrefs}
\usepackage{amssymb}
\usepackage{amsmath}
\usepackage{amsthm}
\usepackage{graphicx}

 \DeclareMathOperator{\Pic}{Pic}
\DeclareMathOperator{\im}{im} 

\theoremstyle{definition}\newtheorem{definition}{Definition}

\newtheorem{remark}[definition]{Remark}
\newtheorem{lemma}[definition]{Lemma}
\newtheorem{proposition}[definition]{Proposition}
\newtheorem{theorem}[definition]{Theorem}

\newtheorem{example}[definition]{Example}
\theoremstyle{remark}\newtheorem*{terminology}{Terminology}
\newtheorem*{subject}{2000 Mathematics Subject Classification}
\newtheorem*{keywords}{Keywords}

\author{Marc Coppens, Johannes Huisman}
\title{Pencils on real curves}
\date{}

\begin{document}
\maketitle \noindent

\begin{abstract}
We consider coverings of real algebraic curves to real rational
algebraic curves. We show the existence of such coverings having
prescribed topological degree on the real locus. From those
existence results we prove some results on Brill-Noether Theory for
pencils on real curves. For coverings having topological degree
$\underline{0}$ we introduce the covering number $k$ and we prove
the existence of coverings of degree 4 with prescribed covering
number.
\end{abstract}

\begin{subject}
14H51; 14P99
\end{subject}

\begin{keywords}
real curve, linear pencil, topological degree, morphism,
Brill-Noether Theory, divisor
\end{keywords}

\section*{Introduction}\label{intro}

In this paper we show the existence of coverings $\pi : X
\rightarrow \mathbb{P}^1$ of $\mathbb{P}^1$ by real curves
satisfying some specific properties related to the real locus
$X(\mathbb{R})$. Part of those constructions is related to recent
results on real pencils (one-dimensional linear systems) on real
curves.

We write $X(\mathbb{C})$ to denote the associated Riemann surface,
$g(X)$ to denote its genus and $a(X)\in \{0,1\}$ such that $2-a(X)$
is the number of connected components of $X(\mathbb{C})\setminus
X(\mathbb{R})$. Consider a covering $\pi : X \rightarrow
\mathbb{P}^1$ and choose an orientation on
$\mathbb{P}^1(\mathbb{R})$. Let $C$ be a connected component of
$X(\mathbb{R})$ and consider the restriction $\pi |_C : C\rightarrow
\mathbb{P}^1(\mathbb{R})$. Fixing an orientation on $C$ this
restriction has a topological degree $\delta _C(\pi )$ and we assume
the orientations are chosen such that $\delta _C(\pi )\geq 0$. We
consider the existence of coverings having prescribed values for
those topological degrees. There are trivial natural restrictions
between $d=\deg (\pi )$, $g=g(X)$, $a=a(X)$ and those topological
degrees and we prove that there exists such a covering $\pi $ for
all values $d$, $g$, $a$ and $\delta _C $ satisfying those
restrictions.

In particular, for each $s\geq 2$, $g \equiv s-1 \pmod{2}$ with
$g\geq s-1$ there exists a real curve $X$ of genus $g$ such that
$a(X)=0$ and $X(\mathbb{R})$ has exactly $s$ connected components
and a morphism $\pi :X\rightarrow \mathbb{P}^1$ of degree $s$ such
that $\delta _C(\pi )=1$ for each component $C$ of $X(\mathbb{R})$.
In particular $\pi (\mathbb{R}):X(\mathbb{R})\rightarrow
\mathbb{P}^1(\mathbb{R})$ is the union of $s$ homeomorphisms. In
\cite {ref9} one introduces very special linear systems on real
curves satisfying strong Clifford-type properties for real curves
and in Proposition 2.1 of \cite {ref9} it is proved that the only
very special pencils are exactly those defined by such coverings. So
we obtain a proof for the existence of very special pencils of all
possible types.

In \cite{ref3} one studies some Brill-Noether properties for pencils
on real curves. We give another proof of part of the results in
\cite{ref3} making intensive use of results coming from the theory
of complex curves. In this way it is enough to prove the existence
of one real curve having a pencil of suited degree. Because of our
existence results we obtain Brill-Noether properties for pencils
having prescribed topological degree on the connected components of
the real locus. While the arguments in \cite{ref3} are restricted to
real curves having real points, our arguments also give
Brill-Noether properties for pencils on real curves without real
points.

The existence and Brill-Noether results are also considered in
\cite{ref23}, \cite{ref17} and \cite{ref18} without considering the
topological degrees. In order to make the arguments more careful
with respect to dimension arguments on the real locus of a variety
defined over $\mathbb{R}$  we use universal spaces representing
morphisms that are known to be globally smooth because of Horikawa
deformation theory.

A real divisor $D$ on a real curve is called totally non-real in
case the support of $D$ contains no real point of $X$. It is known
that for a real linear system on a real curve the parity of the
restriction of a divisor to a connected component of $X(\mathbb{R})$
is constant. Hence if this parity is odd for some component then the
linear system does not contain a totally non-real divisor.
Restricting to linear systems such that the parity is even for all
components of $X(\mathbb{R})$, there is a sharp result in
\cite{ref10} concerning the existence of totally non-real divisors
in such a linear system. For linear systems of large dimension, it
is not easy to find non-trivial examples of such linear systems
having no totally non-real divisors. In the case of pencils our
constructions give a lot of such examples coming from coverings
having only even topological degrees and at least one of them being
non-zero. One can ask for the existence of such examples coming from
coverings with all topological degrees equal to zero. We prove the
existence of all types of such coverings of the smallest possible
degree 4.

In Section 1 we give the constructions that are the base for the
existence of coverings with prescribed topological degrees. In
Section 2 we prove the existence of coverings with prescribed
topological degrees. In Section 3 we discuss Brill-Noether problems
for pencils on real curves. Finally in Section 4 we prove the
existence of coverings of degee 4 with all topological degrees equal
to 0 and having no non-real divisor.

\begin{terminology}
\end{terminology}

A smooth real curve $X$ of genus $g=g(X)$ is a scheme defined over
$\mathbb{R}$ such that the base change $X\times _\mathbb{R}
\mathbb{C}$, denoted by $X_{\mathbb{C}}$, is a complete connected
smooth complex curve of genus $g$. We also write $X(\mathbb{C})$ to
denote the set of closed points of $X_{\mathbb{C}}$ and we consider
it as a Riemann surface of genus $g$. As usual the non-trivial
element of $Gal(\mathbb{C}/\mathbb{R})$ is denoted by $z \rightarrow
\overline{z}$ and it is called (complex) conjugation. For a real
curve $X$ it induces an $\mathbb{R}$-involution $\sigma :
X_{\mathbb{C}} \rightarrow X_{\mathbb{C}}$. For $P\in X(\mathbb{C})$
we write $\overline{P}$ instead of $\sigma (P)$ and we call it the
\emph{conjugated point} of $P$. In case $P=\overline{P}$ then $P$ is
called a \emph{real point} of $X$. The set of real points is called
the real locus $X(\mathbb{R})\subset X(\mathbb{C})$. This locus is a
union of $s(X)$ connected components, each one of them homeomorphic
to a circle. The invariant $a(X)$ is defined by $a(X)=1$ in case
$X(\mathbb{C})\setminus X(\mathbb{R})$ is connected and otherwise
$a(X)=0$ (in that case $X(\mathbb{C})\setminus X(\mathbb{R})$ has
two connected components). We say a smooth real curve $X$ has
\emph{topological type} $(g,s,a)$ if $g(X)=g$, $s(X)=s$ and
$a(X)=a$. By a theorem of Weichold such real curve exists if and
only if either $a=1$ and $0\leq s\leq g$ else $a=0$, $s \equiv g+1
\pmod 2$ and $1\leq s\leq g+1$ (see \cite{ref16}, see also
\cite{ref1}*{Theorem 5.3}). A triple $(g,s,a)$ satisfying those
restrictions is called an \emph{admissable topological type} for
real curves. In case $P\neq \overline{P}$ then we call
$P+\overline{P}$ a \emph{non-real point} of $X$.

For a divisor $D=\sum _{i=1}^n m_iP_i$ on $X_{\mathbb{C}}$ we write
$\overline{D}=\sum _{i=1}^n m_i\overline{P_i}$ and we call it the
conjugated divisor of $D$. As usual $\sum _{i=1}^n m_i$ is called
the degree of $D$, denoted by $\deg (D)$. We say $D$ is a real
divisor if $\overline{D}=D$. If we say $D$ is a divisor on $X$ then
we mean it is a \emph{real divisor} on $X_{\mathbb{C}}$.

Let $L$ be an invertible sheaf on $X_{\mathbb{C}}$. It can be
described by means of trivializations on an open covering $\left(
U_i \right)_{i\in I}$ of $X(\mathbb{C})$ and transition functions
$a_{i,j}$ (those are regular functions $U_i \cap U_j \rightarrow
\mathbb{C}^*$). The conjugated invertible sheaf $\overline{L}$ is
defined by means of trivializations on the open covering $\left(
\sigma (U_i) \right)_{i\in I}$ and transition functions
$\overline{a_{i,j}}$ (for $P\in \sigma (U_i \cap U_j)$ one has
$\overline{a_{i,j}}(P)=\overline{a_{i,j}(\overline{P})}$). Each
invertible sheaf $L$ on $X_{\mathbb{C}}$ is isomorphic to
$\mathcal{O}_{X_{\mathbb{C}}}(D)$ for some divisor $D$ on
$X_{\mathbb{C}}$. In case $k=\deg (D)$ then we say $L$ is an
invertible sheaf of degree $k$ and we write $\deg (L)=k$. In case
$L\cong \mathcal{O}_X(D)$ then $\overline{L}\cong
\mathcal{O}_X(\overline{D})$.

We say $L$ is a \emph{real invertible sheaf} if $L \cong
\mathcal{O}_X(D)$ for some real divisor $D$. We say $L$ is
\emph{invariant under conjugation} if $L \cong \overline{L}$. In
particular a real invertible sheaf is invariant under conjugation.
Let $\Pic (X)$ be the Picard scheme of $X$ defined over $\mathbb{R}$
(see \cite{ref11a}). Then $\Pic (X)_{\mathbb{C}}$ represents the
Picard functor on $X_{\mathbb{C}}$, in particular $\Pic
(X)(\mathbb{C})$ parameterizes invertible sheaves on
$X_{\mathbb{C}}$. For $k\in \mathbb{Z}$ one has natural subschemes
$\Pic ^k(X)$ such that $\Pic ^k(X)(\mathbb{C})$ parameterizes
invertible sheaves of degree $k$ on $X_{\mathbb{C}}$. The real locus
$\Pic ^k(X)(\mathbb{R})$ parameterizes invertible sheaves of degree
$k$ on $X_{\mathbb{C}}$ invariant under conjugation. Let $\Pic
^k(X)(\mathbb{R})^+$ be the sublocus parameterizing real invertible
sheaves. In case $X(\mathbb{R})\neq \emptyset$ then $\Pic (X)$ also
represents the Picard functor on $X$ and therefore
$\Pic^k(X)(\mathbb{R})^+=\Pic^k(X)(\mathbb{R})$ in that case. In
case $X(\mathbb{R})=\emptyset$ then $\Pic^k(X)(\mathbb{R})^+$ is a
subgroup of $\Pic^k(X)(\mathbb{R})$ with quotient
$\mathbb{Z}/2\mathbb{Z}$. For $L\in \Pic^k(X)(\mathbb{R})$ with
$L\notin \Pic(X)(\mathbb{R})^+$ one has $\deg (L) \equiv g-1
\pmod{2}$ (see \cite{ref2}*{Proposition 2.2}).

We write $|L|$ to denote the complete linear system of an invertible
sheaf $L$ on $X_{\mathbb{C}}$. In case $L$ is a real invertible
sheaf then $|L|(\mathbb{R})$ is the space of real divisors contained
in $|L|$. One has $|L|=|L| (\mathbb{R}) \otimes \mathbb{C}$ in a
natural way (meaning $|L| (\mathbb{R})$ can be considered as a
projective space $\mathbb{P}^r$ and then
$|L|=\mathbb{P}^r_{\mathbb{C}}$). As usual we write $g^r_d$ to
denote a linear system of dimension $r$ and degree $k$ on
$X_{\mathbb{C}}$. In case $L$ is a real invertible sheaf of degree
$k$ then we say a linear subsystem $g^r_k$ of $|L|$ is defined over
$\mathbb{R}$ if there exists a $k$-dimensional linear subsystem
$g^r_k(\mathbb{R})\subset |L|(\mathbb{R})$ such that
$g^r_k=g^r_k(\mathbb{R}) \otimes \mathbb{C}$. In such case we say
$g^r_k$ is a linear system on $X$. In particular if $L$ is a real
invertible sheaf then we say $|L|$ is a linear system on $X$.

We write $\mathbb{P}^1$ to denote the projective line defined over
$\mathbb{R}$ and we write $R_0$ to denote the smooth real curve of
genus 0 without real points (defined by the equation $X^2+Y^2+Z^2=0$
in $\mathbb{P}^2$).

Let $X$ be a smooth real curve. A covering $\pi : X\rightarrow
\mathbb{P}^1$ or $\pi : X \rightarrow R_0$ is a finite morphism
defined over $\mathbb{R}$. We say $\pi$ has degree $d$ in case the
associated covering $\pi _{\mathbb{C}}: X _{\mathbb{C}} \rightarrow
\mathbb{P}^1_{\mathbb{C}}$ has degree $d$. Such covering corresponds
to a base point free pencil $g^1_d$ on $X_{\mathbb{C}}$. In case of
$\pi :X \rightarrow \mathbb{P}^1$ this is called a \emph{real
pencil} on $X$. In case of $\pi : X \rightarrow R_0$ this is called
a \emph{non-real invariant pencil} on $X$. This case only can occur
if $X(\mathbb{R})=\emptyset$ ($X$ has no real points) and moreover,
in this case the pencil has no real divisor but it is invariant
under complex conjugation. We write $\pi (\mathbb{C}):X(\mathbb{C})
\rightarrow \mathbb{P}^1(\mathbb{C})$ to denote the covering of
Riemann surfaces. In case $X(\mathbb{R})$ is not empty we write $\pi
(\mathbb{R}): X(\mathbb{R}) \rightarrow \mathbb{P}^1(\mathbb{R})$ to
denote the restriction of $\pi (\mathbb{C})$ to the real locus
$X(\mathbb{R})$. For a connected component $C$ of $X(\mathbb{R})$ we
write $\pi _C:C \rightarrow \mathbb{P}^1(\mathbb{R})$ to denote the
restriction of $\pi (\mathbb{R})$ to $C$. On
$\mathbb{P}^1(\mathbb{R})$, homeomorphic to a circle $S^1$, we
choose an orientation and then for $C$, also homeomorphic to $S^1$,
we choose the orientation such that the degree of $\pi _C$, denoted
by $\delta _C(\pi)$, is nonnegative. If $P\in C$ and $\pi _C$ is not
ramified at $P$, then the local degree $\delta _P(\pi )$ is equal to
1 (resp. -1) if $\pi _C$ preserves (resp. reverses) the orientation
locally at $P$. Let $C_1, \cdots , C_s$ be the connected components
of $X(\mathbb{R})$ and let $\delta _i=\delta _{C_i}(\pi)$ for $1\leq
i\leq s$. We can always assume $\delta _1 \geq \delta _2 \geq \cdots
\geq \delta _s$ and then we say $\pi$ has \emph{topological degree}
$(\delta _1, \delta _2, \cdots, \delta _s)$ (shortly denoted by
$\underline{\delta}$).

\section{Constructions}\label{section1}

The proof of the existence of coverings with prescribed topological
degrees uses an induction argument. To make that argument we start
with a smooth real curve $Y$ and a suited covering $\pi _Y: Y
\rightarrow \mathbb{P}^1$. Using this covering we construct a
covering $\pi _0: X_0 \rightarrow \mathbb{P}^1$ with $X_0$ being a
real singular nodal curve and we use a real smoothing $\pi _t: X_t
\rightarrow \mathbb{P}^1$ of $\pi _0$. In this part we describe
those constructions giving relations between topological degrees of
$\pi_Y$ and of $\pi_t$. This will be the base for the induction
argument in the next section.

In those constructions we start by taking local smoothings of the
nodes of $X_0(\mathbb{C})$ having a natural antiholomorphic
involution. Those local smoothings glue with the complement $V$ of a
neighborhood of the nodes of $X_0(\mathbb{C})$ giving rise to a
deformation of compact Riemann surfaces $X_t(\mathbb{C})$ and
holomorphic coverings $\pi _t(\mathbb{C}): X_t(\mathbb{C})
\rightarrow \mathbb{P}^1(\mathbb{C})$. This complement $V$ can be
taken to be invariant under complex conjugation on $X_0(\mathbb{C})$
and this fits with the antiholomorphic involution of the local
smoothings under the gluing. Hence we obtain an antiholomorphic
involution $\sigma _t$ on $X_t(\mathbb{C})$. It is well-known that
this defines a smooth real curve $X_t$ inducing the Riemann surface
$X_t(\mathbb{C})$ such that $\sigma _t$ corresponds to complex
conjugation (see e.g. \cite {ref1}*{Section 4}). Moreover the
morphism $\pi _t(\mathbb{C}):X_t(\mathbb{C}) \rightarrow
\mathbb{P}^1(\mathbb{C})$ is invariant under complex conjugation
hence it comes from a morphism $\pi _t:X_t \rightarrow \mathbb{P}^1$
(indeed, the graph of $\pi _t(\mathbb{C})$ is a closed subspace of
$(X_t\times \mathbb{P}^1)(\mathbb{C})$ invariant under complex
conjugation on $X_t \times \mathbb{P}^1$).

\subsection{Construction I}\label{subsection1.1}

Let $Y$ be a real curve of genus $g$ and let $\pi :Y \rightarrow
\mathbb{P}^1$ be a morphism of degree $k$ defined over $\mathbb{R}$.
Assume $C$ is a connected component of $Y(\mathbb{R})$ and let $P$
be a point of $C$ such that $\delta _P(\pi )=1$ (such point $P$ does
exist because $\delta_C(\pi)\geq 0$). Take a copy of $\mathbb{P}^1$
and consider the singular curve $X_0=Y\cup _P\mathbb{P}^1$ obtained
by identifying $P$ on $X$ with $\pi (P)$ on $\mathbb{P}^1$. This
singular curve has a natural morphism $\pi _0$ defined over
$\mathbb{R}$ of degree $k+1$ to $\mathbb{P}^1$ having restriction
$\pi$ to $Y$ and the identity to $\mathbb{P}^1$ (see Figure
\ref{Figure 1}).

\begin{figure}[h]
\begin{center}
\includegraphics[height=3 cm]{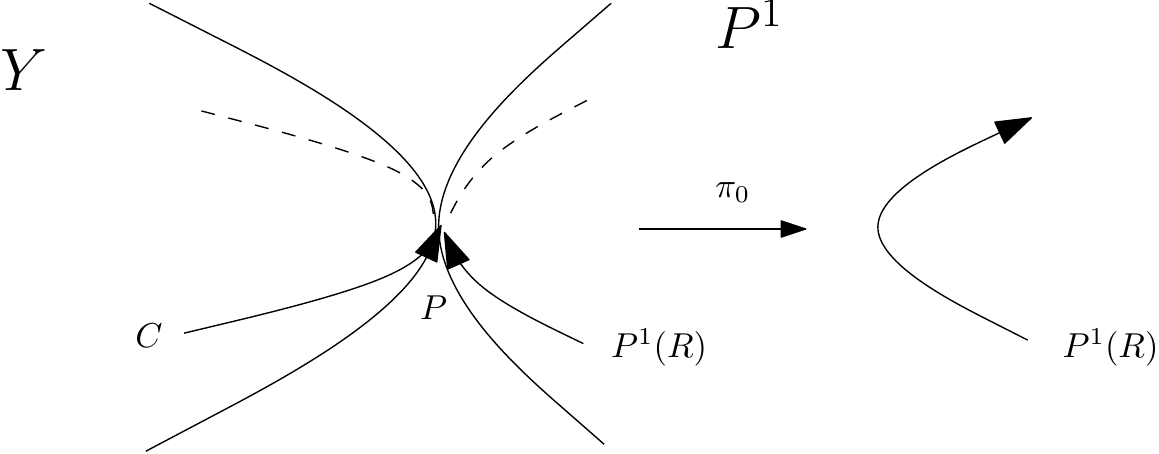}
\caption{Construction I, the curve $X_0$} \label{Figure 1}
\end{center}
\end{figure}

Locally at $P$ this situation can be described inside $\mathbb{C}^2$
such that the curve $X_0$ has equation $x^2-y^2=0$ and the morphism
is locally defined by $(x,y)\rightarrow x$ and we can assume the
coordinates are compatible with the real structure on $X_0$ (this
means complex conjugation on $X_0$ corresponds to complex
conjugation in  $\mathbb{C}^2$). Let $U$ be a small neighborhood of
$(0,0)$ in $\mathbb{C}^2$ and $V\subset U$ a much smaller one and
let $U_0=U\cap X_0$, $V_0=V\cap X_0$. For $Q\neq P$ in $U_0$ we can
use $x$ to define a holomorphic coordinate at $Q$. Consider a local
deformation of $X_0$ at $P$ defined by $x^2-y^2=t$ with $t\in
\mathbb{R}$ ($|t|$ very small) and let $U_t$ (resp. $V_t$) be the
intersection with $U$ (resp. $V$). We use a gluing of $U_t$ and
$X_0(\mathbb{C})\setminus V_0$ as follows. For $z_0\in \mathbb{C}^*$
let $_{z_0} \sqrt{z}$ be the locally defined holomorphic square root
function such that $_{z_0} \sqrt{z_0^2}=z_0$. A point $Q\in
U_0\setminus V_0$ has coordinates $(x,x)$ or $(x,-x)$. We identify
$(x,x)$ with $(x,_x\sqrt{x^2-t})$ and $(x,-x)$ with
$(x,-_x\sqrt{x^2-t})$. (One should adapt the description of $V_t$
and $U_t$ to this identification.) This defines a Riemann surface
$X_t(\mathbb{C})$. On $U_t$ we define $\sigma
_t(x,y)=(\overline{x},\overline{y})$. We need to show this behaves
well under the previous identification. Since $_x\sqrt{x^2-t}$ is
close to $x$ one has $\overline{_x\sqrt{x^2-t}}$ is close
$\overline{x}$. Moreover $\left(\overline{ _x\sqrt{x^2-t}}
\right)^2=\overline{\left( _x\sqrt{x^2-t}
\right)^2}=\overline{x^2-t}=\overline{x}^2-t$ hence
$\overline{_x\sqrt{x^2-t}}=_{\overline{x}}\sqrt{\overline{x}^2-t}$.
Hence under the identification $(\overline{x},\overline{x})$ is
identified with $(\overline{x},\overline{_x\sqrt{x^2-t}})$ and
similarly $(\overline{x},-\overline{x})$ is identified with
$(\overline{x},\overline{-_{x}\sqrt{x^2-t}})$. As mentioned at the
beginning of this section we obtain a real smooth curve $X_t$ and a
covering $\pi _t: X_t\rightarrow \mathbb{P}^1$ of degree $k+1$.
Under this deformation the union $C\cup_P\mathbb{P}^1(\mathbb{R})$
deforms to a connected component $C_t$ of $X_t(\mathbb{R})$

In case $t>0$ the morphism $\pi_t$ has ramification on $C_t$ above
$x=\pm \sqrt{t}$. We call this \emph{the deformation with real
ramification}. In order to have an orientation on $C_t$ we have to
change the orientation on the attached $\mathbb{P}^1(\mathbb{R})$.
This implies that $\delta_{C_t}(\pi_t)=\delta_C(\pi)-1$. Of course,
if $\delta_C(\pi)=0$ we also change the orientation of $C_t$ in
order to obtain $\delta_{C_{t}}(\pi_t)=1$ (see Figure \ref{Figure
2}).

\begin{figure}[h]
\begin{center}
\includegraphics[height=3 cm]{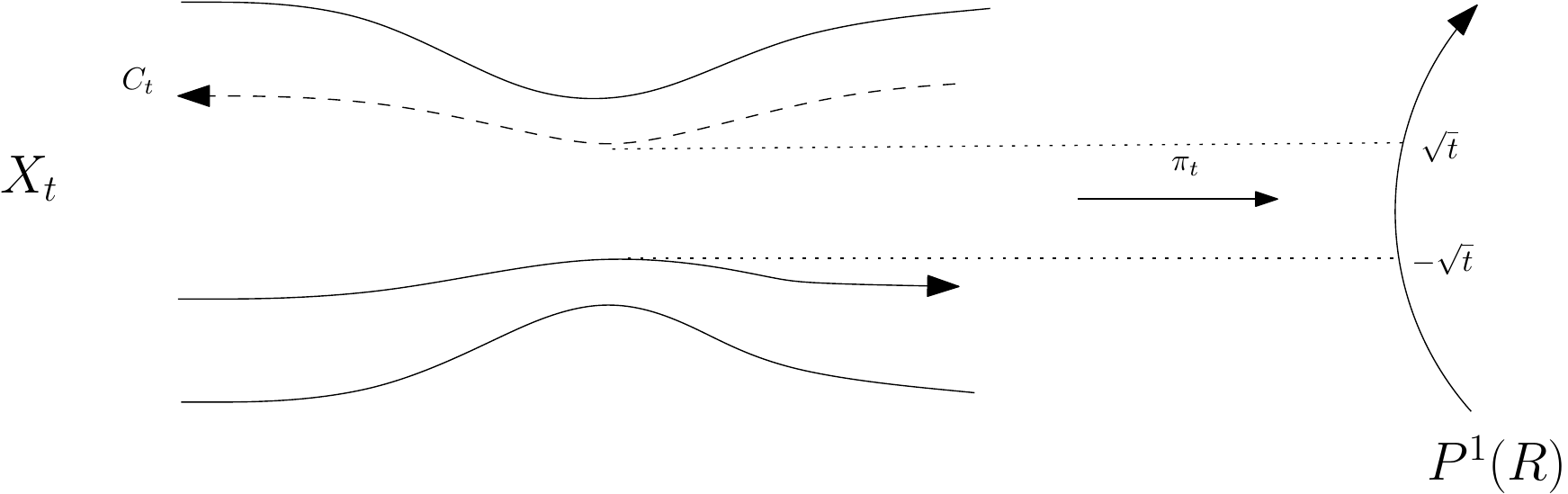}
\caption{Construction I, deformation with real ramification}
\label{Figure 2}
\end{center}
\end{figure}

In case $t<0$ the morphism $\pi_t$ has no ramification on $C_t$
close to $P$. We call this \emph{the deformation without real
ramification}. We use the orientation on $C_t$ obtained from both
the orientation on $C$ and the attached $\mathbb{P}^1(\mathbb{R})$.
This implies $\delta_{C_t}(\pi_t)=\delta_C(\pi)+1$ (see Figure
\ref{Figure 3}).

\begin{figure}[h]
\begin{center}
\includegraphics[height=3 cm]{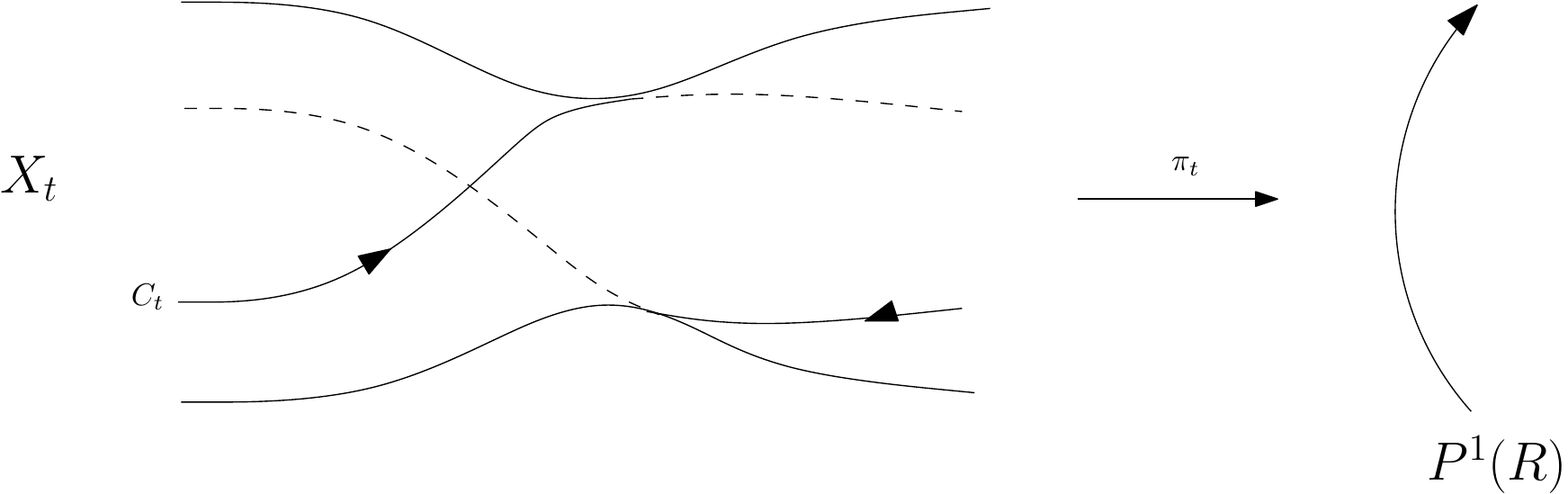}
\caption{Construction I, deformation without real ramification}
\label{Figure 3}
\end{center}
\end{figure}

In both cases $X_t(\mathbb{C})\backslash X_t(\mathbb{R})$ is
connected if and only if $X(\mathbb{C})\backslash X(\mathbb{R})$ is
connected.

\subsection{Construction II}\label{subsection1.2}

Let $Y$ be a real curve of genus $g$ and let $\pi :Y\rightarrow
\mathbb{P}^1$ be a morphism of degree $k$ defined over $\mathbb{R}$.
Let $Q\in \mathbb{P}^1(\mathbb{R})$ not a branch point of $\pi$ and
assume $P+\overline{P}\subset\pi^{-1}(Q)$, a non-real point on $Y$
(in particular we assume such non-real point exists). Let $X_0$ be
the singular curve obtained from $X$ by identifying $P$ with
$\overline{P}$. This singular curve has a natural morphism of degree
$k$ to $\mathbb{P}^1$ defined over $\mathbb{R}$ obtained from $\pi$
(see Figure \ref{Figure 4}).

\begin{figure}[h]
\begin{center}
\includegraphics[height=3 cm]{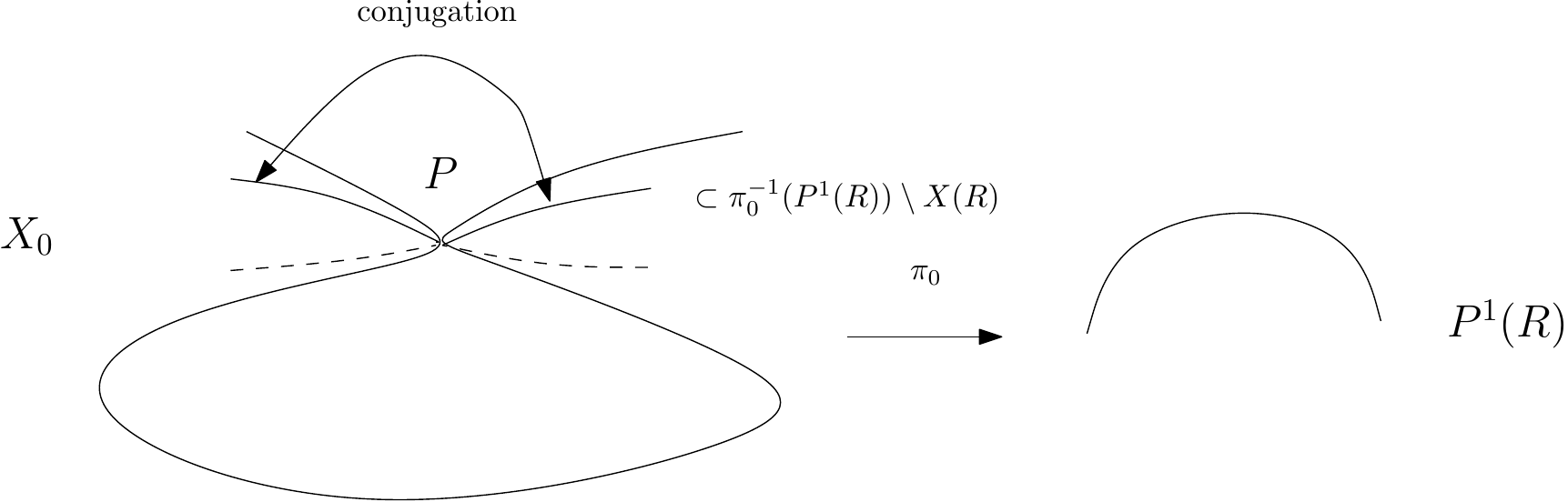}
\caption{Construction II, the curve $X_0$}\label{Figure 4}
\end{center}
\end{figure}

On $X_0$ the singularity (again denoted by $P$) is an isolated real
point. Locally at $P$ this situation can be described inside
$\mathbb{C}^2$ such that the curve has equation $x^2+y^2=0$ and the
morphism is locally defined by $(x,y)\rightarrow x$ with coordinates
compatible with the real structure on $X_0$. We use $U$ and $V$ as
in Construction I. Consider the local deformation of $X_0$ at $P$
defined by $x^2+y^2=t$ with $t\in \mathbb{R}$ ($|t|$ very small). A
point $Q\in U_0\setminus V_0$ has coordinates $(x,ix)$ or $(x,-ix)$
and $\sigma (Q)$ corresponds to resp.
$(\overline{x},-i\overline{x})$ or $(\overline{x},i\overline{x})$.
We identify $(x,ix)$ (resp. $(x,-ix))$) on $U_0 \setminus V_0$ with
$(x,_{ix}\sqrt{-x^2+t})$ (resp. $(x,-_{ix}\sqrt{-x^2+t})$) on $U_t$.
This defines a Riemann surface $X_t(\mathbb{C})$. On $U_t$ we define
$\sigma _t(x,y)=(\overline{x},\overline{y})$ and we check that it
behaves well under the identification. Consider
$Q=(x,_{ix}\sqrt{-x^2+t})$ on $U_t\setminus V_t$. Since
$_{ix}\sqrt{-x^2+t}$ is close to $ix$ one has
$\overline{_{ix}\sqrt{-x^2+t}}$ is close to $-i\overline{x}$, hence
$\overline{_{ix}\sqrt{-x^2+t}}=-_{i\overline{x}}\sqrt{-\overline{x}^2+t}$.
Hence $\sigma _t(Q)$ is identified with
$(\overline{x},-i\overline{x})=\sigma _0(x,ix)$. We obtain a smooth
real curve $X_t$ of genus $g+1$ together with a morphism
$\pi_t:X_t\rightarrow \mathbb{P}^1$ of degree $k$ defined over
$\mathbb{R}$.

In case $t>0$ there is a new component $C_t$ of $X_t(\mathbb{R})$
close to $P$ and under $\pi_t$ this maps to the interval
$[-\sqrt{t},\sqrt{t}]$. In particular $\delta_{C_t}(\pi_t)=0$. We
call this \emph{the deformation with real ramification}. In this
case $X_t(\mathbb{C})\backslash X_t(\mathbb{R})$ is connected if and
only if $X(\mathbb{C})\backslash X(\mathbb{R})$ is connected (see
Figure \ref{Figure 5}).

\begin{figure}[h]
\begin{center}
\includegraphics[height=4 cm]{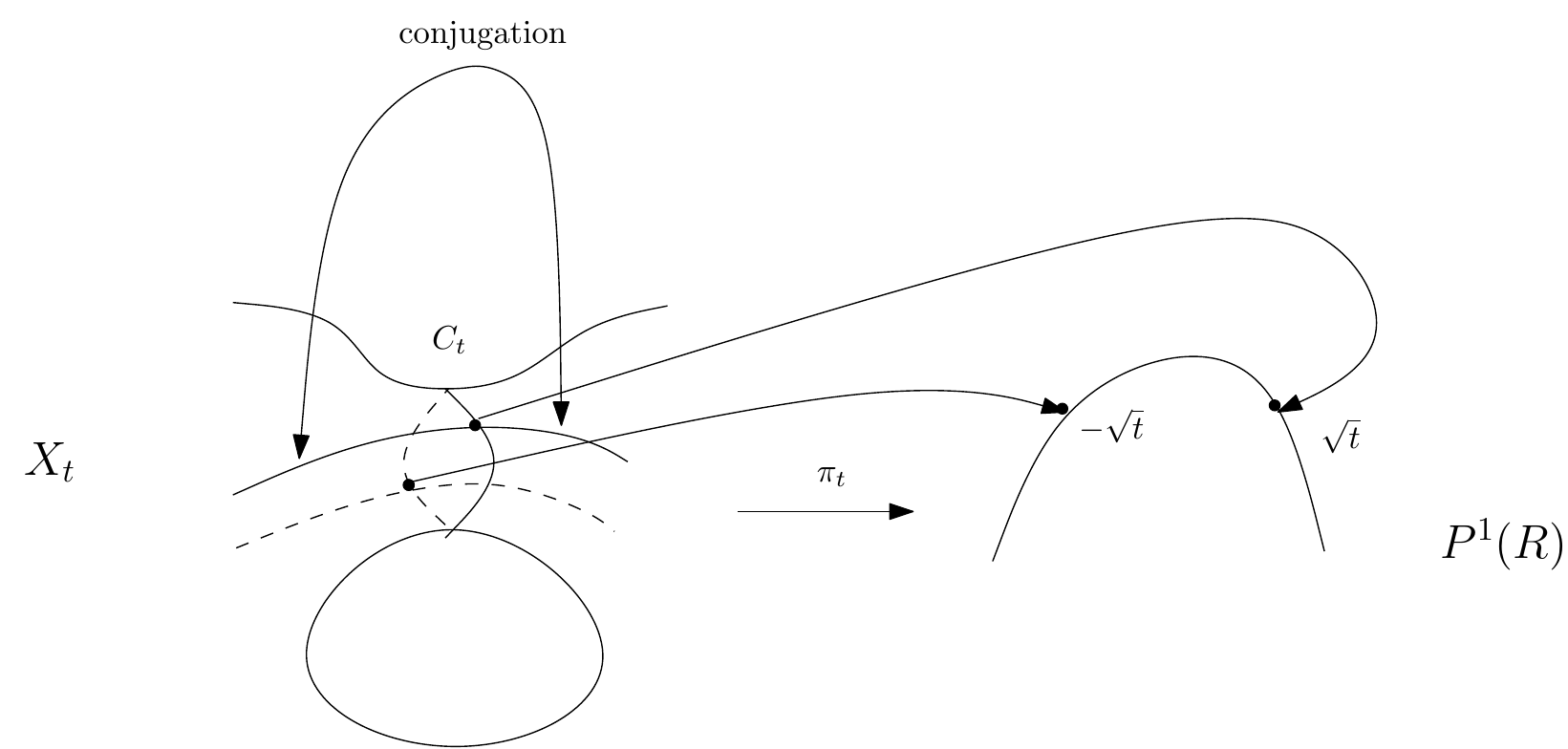}
\caption{Construction II, deformation with real ramification}
\label{Figure 5}
\end{center}
\end{figure}

In case $t<0$ then $X_t(\mathbb{R})$ contains no real point close to
$P$. We call this the deformation without real ramification. In this
case $X_t(\mathbb{C})\backslash X_t(\mathbb{R})$ is always connected
(see Figure \ref{Figure 6}).

\begin{figure}[h]
\begin{center}
\includegraphics[height=4 cm]{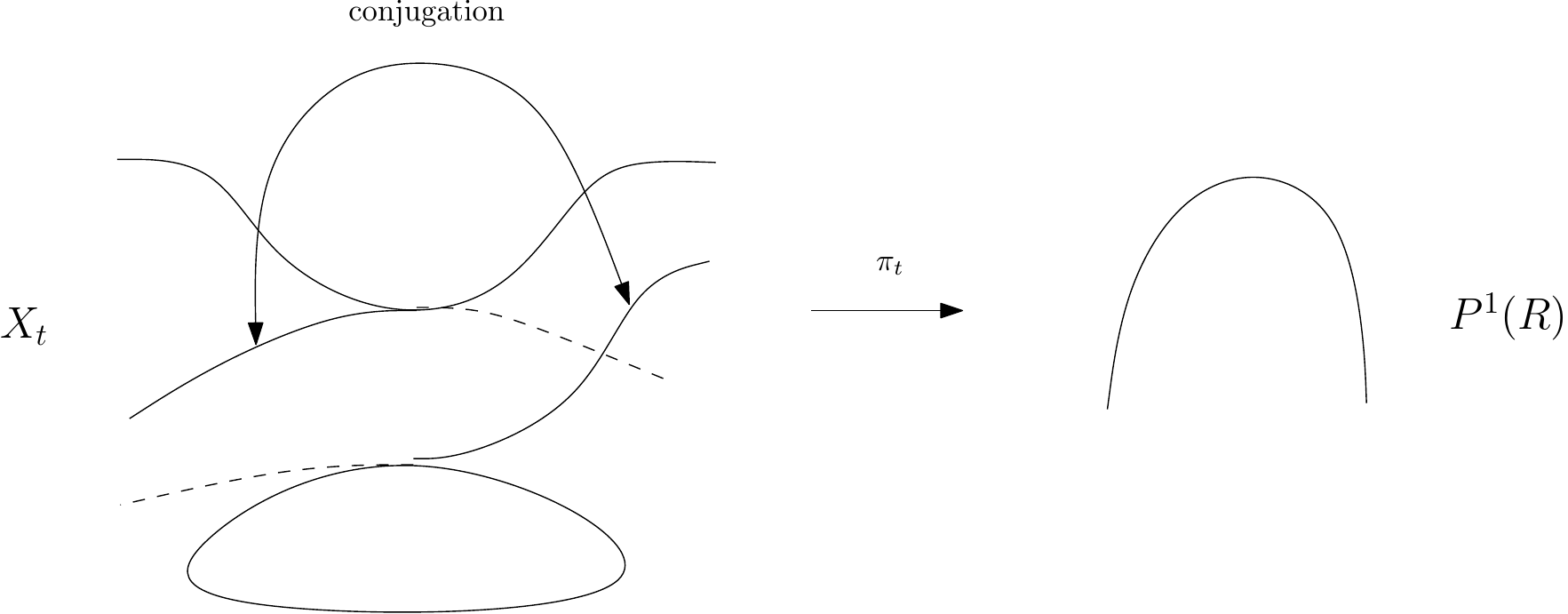}
\caption{Construction II, deformation without real ramification}
\label{Figure 6}
\end{center}
\end{figure}

\subsection{Construction III}\label{subsection1.3}

Let $Y$ be a real curve of genus $g$ and let $\pi :Y\rightarrow
\mathbb{P}^1$ be a morphism of degree $k$ defined over $\mathbb{R}$.
Assume $P+\overline{P}$ is a non-real point of $Y$ such that $\pi
(P)\neq \pi (\overline{P})$ and $P$, $\overline{P}$ are not
ramification points of $\pi _{\mathbb{C}}$. Now we let $X_0=Y\cup
_{P+\overline{P}}\mathbb{P}^1$ be the union of $Y$ and
$\mathbb{P}^1$ identifying $P$ with $\pi (P)$ (still denoted by $P$)
and $\overline{P}$ with $\pi (\overline{P})$ (still denoted by
$\overline{P}$). This singular curve has a natural morphism $\pi _0$
defined over $\mathbb{R}$ of degree $k+1$ to $\mathbb{P}^1$ having
restriction $\pi$ to $Y$ and the identity to $\mathbb{P}^1$ (see
Figure \ref{Figure 7}).

\begin{figure}[h]
\begin{center}
\includegraphics[height=4 cm]{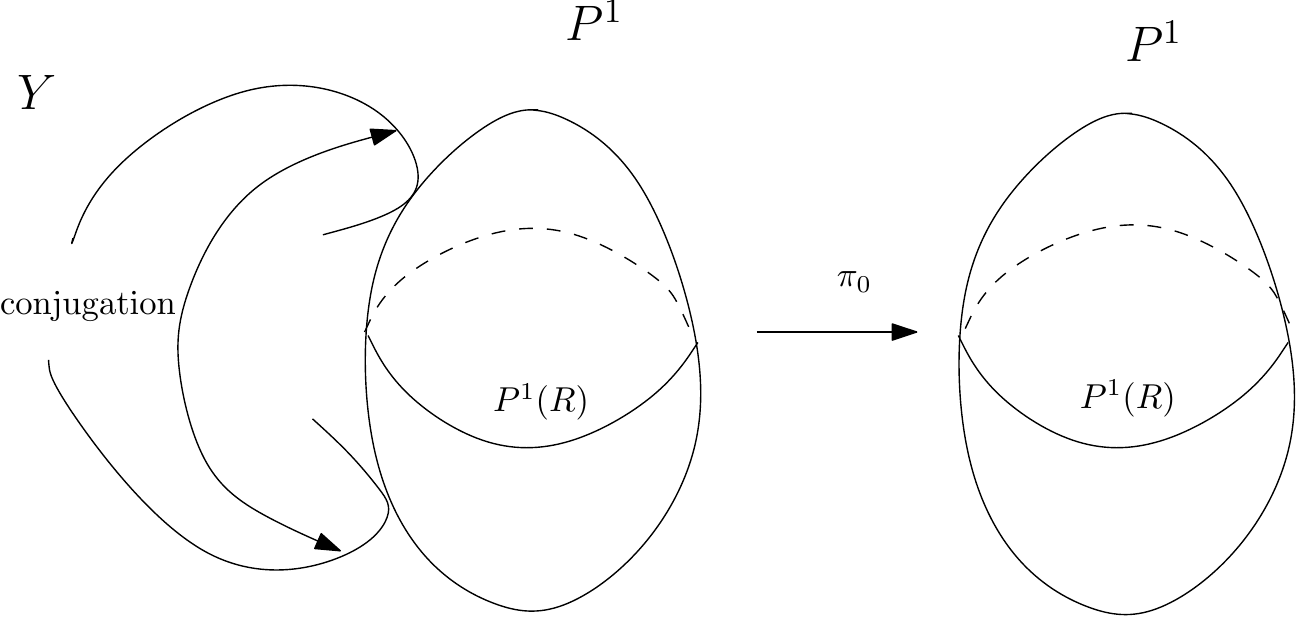}
\caption{Construction III, the curve $X_0$} \label{Figure 7}
\end{center}
\end{figure}

Choose a local deformation over $\mathbb{C}$ making it smooth at $P$
and take the complex conjugated local deformation at $\overline{P}$
and glue it with the remaining part of $X_0$. We obtain a smooth
real curve $X_t$ of genus $g+1$ together with a morphism
$\pi_t:X_t\rightarrow \mathbb{P}^1$ of degree $k+1$ defined over
$\mathbb{R}$. One has $g(X_t)=g+1$ and $\mathbb{P}^1(\mathbb{R})$ on
$X\cup _{P+\overline{P}}\mathbb{P}^1$ deforms to a connected
component $C_t$ of $X_t(\mathbb{R})$ such that
$\delta_{C_t}(\pi_t)=1$. For this construction
$X_t(\mathbb{C})\backslash X_t(\mathbb{R})$ is connected if and only
if $X(\mathbb{C})\backslash X(\mathbb{R})$ is connected (see Figure
\ref{Figure 8}).

\begin{figure}[h]
\begin{center}
\includegraphics[height=4 cm]{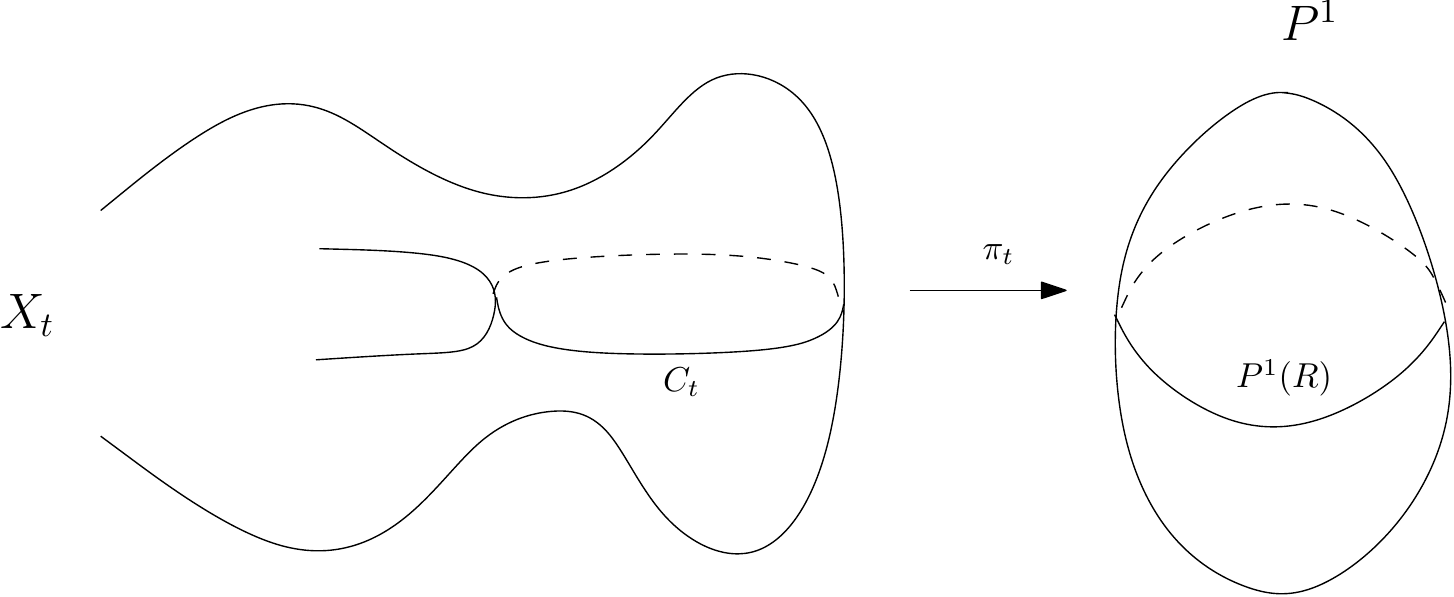}
\caption{Construction III, deformation} \label{Figure 8}
\end{center}
\end{figure}

\subsection{Construction IV}\label{subsection1.4}

Consider a double covering $\tau_0:R_0\rightarrow \mathbb{P}^1$
defined over $\mathbb{R}$. Assume $Y$ is a real curve without real
points and $\pi:Y\rightarrow \mathbb{P}^1$ is a morphism of degree
$k$ defined over $\mathbb{R}$ (in particular $k$ is even). Choose
$P+\overline{P}$ on $\mathbb{P}^1$ (a non-real point) such that $P$,
$\overline{P}$ are not branch points of $\pi _{\mathbb{C}}$ and
choose non-real points $P_0+\overline{P_0}$ on $R_0$ and
$Q+\overline{Q}$ on $Y$ such that
$\tau_0(\mathbb{C})(P_0)=\pi(\mathbb{C})(Q)=P$. Consider the nodal
curve $X_0=Y\cup_{P+\overline{P}}R_0$ obtained by identifying $P_0$
to $Q$ (still denoted by $Q$) and $\overline{P_0}$ to $\overline{Q}$
(still denoted by $\overline{Q}$). This is a singular curve defined
over $\mathbb{R}$ without real points and $\pi$ together with
$\tau_0$ induce a morphism of that singular curve to $\mathbb{P}^1$
defined over $\mathbb{R}$. Take a local deformation at $Q$ over
$\mathbb{C}$ smoothing the singularity, take the complex conjugated
deformation at $\overline{Q}$ and glue it with the remaining part of
$X_0$ to obtain a smooth real curve $X_t$ of genus $g+1$ together
with a morphism $\pi_t:X_t\rightarrow \mathbb{P}^1$ of degree $k+2$.
This smooth real curve does not have any real point.

\subsection{Construction V}\label{subsection1.5}

Assume $Y$ is a real curve of genus $g$ without real points and
assume $\pi :Y\rightarrow R_0$ is a morphism of degree $k$ defined
over $\mathbb{R}$. Choose a non-real point $P_0+\overline{P_0}$ on
$R_0$ such that $P_0$, $\overline{P_0}$ are not branch points of
$\pi _{\mathbb{C}}$ and a non-real point $P+\overline{P}$ on $Y$
with $\pi (P)=P_0$. Consider the nodal curve $X_0=Y\cup _{P+P_0}R_0$
obtained by identifying $P_0$ to $P$ (still denoted by $P$) and
$\overline{P_0}$ to $\overline{P}$ (still denoted by
$\overline{P}$). This is a singular curve defined over $\mathbb{R}$
without real points and $\pi$ together with the identity $1_{R_0}$
induces a morphism of degree $k+1$ of that singular curve to $R_0$
defined over $\mathbb{R}$. Take a local deformation at $P$ over
$\mathbb{C}$ smoothing the singularity, take the complex conjugated
deformation at $\overline{P}$ and glue it with the remaining part of
$X_0$ to obtain a smooth real curve $X_t$ of genus $g+1$ together
with a morphism $\pi _t:X_t\rightarrow R_0$ of degree $k+1$.

\section{Existence of real curves having pencils with prescribed
topological properties}\label{section2}

Let $X$ be a real curve of topological type $(g,s,a)$ and let $\pi
:X\rightarrow \mathbb{P}^1$ be a morphism of degree $k$ defined over
$\mathbb{R}$. Assume $s\geq 1$ and let
$\underline{\delta}=(\delta_1, \cdots , \delta_s)$ be the
topological degree of $\pi$. Of course we need $\sum_{i=1}^s\delta
_i\leq k$. For each $s\in \mathbb{P}^1(\mathbb{R})$ the number of
real points in $\pi ^{-1}(s)$ counted with multiplicity has the same
parity as $\sum_{i=1}^s\delta _i$. Since the non-real points of $X$
are pairs of conjugated points on $X(\mathbb{C})$ it follows $k-\sum
_{i=1}^s \delta _i \equiv 0 \pmod{2}$. In case $\sum _{i=1}^s \delta
_i=k$ then for each $s\in \mathbb{P}^1(\mathbb{R})$ the fiber $\pi
^{-1}(s)$ has exactly $\delta _i$ different points on $C_i$. Since
some fiber needs to contain points of $C_s$ one has $\delta _s\geq
1$ in that case, hence $\sum_{i=1}^s \delta _i \leq k-2$ in case
$\delta _s=0$ (see also \cite{ref15}).

\begin{definition}\label{definition1}
We say $\delta _1\geq \cdots \geq \delta _s\geq 0$ is an
\emph{admissible topological degree for morphisms of degree $k$} if
\begin{itemize}
\item $\sum_{i=1}^s \delta _i\leq k$,
\item $k-\sum_{i=1}^s \delta _i \equiv 0 \pmod{2}$,
\item $\sum_{i=1}^s \delta _i\leq k-2$ in case $\delta _s=0$.
\end{itemize}
\end{definition}

For the next theorem we need some more restriction. In case there is
a morphism $\pi :X\rightarrow \mathbb{P}^1$ of degree $k$ with $\sum
_{i=1}^s \delta _i =k$ then by definition $\pi ^{-1} \left(
\mathbb{P}^1(\mathbb{R}) \right)=X(\mathbb{R})$, hence
$X(\mathbb{C})\setminus X(\mathbb{R})=\pi ^{-1} \left(
\mathbb{P}^1(\mathbb{C}) \setminus \mathbb{P}^1 (\mathbb{R})
\right)$. Since $\mathbb{P}^1(\mathbb{C}) \setminus
\mathbb{P}^1(\mathbb{R})$ is disconnected it follows $X(\mathbb{C})
\setminus X(\mathbb{R})$ is disconnected too, hence $a(X)=0$.
Therefore, in case $a(X)=1$ one always has $\sum _{i=1}^s \delta _i
\leq k-2$.

\begin{theorem}\label{theorem2}
Let $k\geq 3$, let $(g,s,a)$ be an admissable topological type of
real curves and let $\underline{\delta}$ be an admissible
topological degree for morphisms of degree $k$. Assume $\sum
_{i=1}^s \delta _i \leq k-2$ in case $a=1$. There exists a real
curve $X$ of topological type $(g,s,a)$ having a morphism $\pi
:X\rightarrow \mathbb{P}^1$ of degree $k$ of topological degree
$\underline{\delta}$.
\end{theorem}

\begin{remark}\label{remarktheorem2}
The case $k=2$ is contained in e.g. \cite{ref2}*{Section 6}. In that
case there are more restrictions in the case $a=0$. Nevertheless,
the existence of real hyperelliptic cruves is the starting point in
a lot of cases of the proof of the theorem, even in Case 5 which is
excluded for real hyperelliptic curves unless $s=g+1$.
\end{remark}

Section \ref{subsection2.1} is the proof of Theorem \ref{theorem2}
in case $a=1$ and Section \ref{subsection2.2} is the proof of
Theorem \ref{theorem2} in case $a=0$. In Section \ref{subsection2.3}
we prove the corresponding theorem for morphisms to $R_0$.

\subsection{The case $a=1$}\label{subsection2.1}

First we consider the case $s\neq 0$. Let $s'$ be the number of
$\delta _i\neq 0$ and let $\widetilde{X}$ be a real hyperelliptic
curve of topological type $(g-s',s-s',1)$ (such curves do exist, see
\cite{ref2}*{Section 6}). Let $\widetilde{\pi}$ be the associated
double covering of $\mathbb{P}^1$ then it follows that
$\delta_{\widetilde{C}_i}(\widetilde{\pi})=0$ for $1\leq i\leq
s-s'$. Applying Construction III $s'$ times we obtain
$\pi_1:X_1\rightarrow \mathbb{P}^1$ with $\deg (\pi_1)=2+s'$;
$g(X_1)=g$ and $X_1(\mathbb{R})$ has $s$ components $C_1^1,
\cdots,C_s^1$ such that $\delta_{C_i^1}(\pi_1)=1$ for $1\leq i\leq
s'$ and $\delta_{C_i^1}(\pi_1)=0$ for $s'+1\leq i\leq s$, in
particular $X_1$ has topological type $(g,s,1)$. For each $1\leq
i\leq s'$ we use Construction I $\delta_i-1$ times on $C^1_i$ and
using the deformation without real ramification. We obtain $\pi_2:
X_2\rightarrow \mathbb{P}^1$ with $\deg
(\pi_2)=2+\sum_{i=1}^{s'}\delta_i$, $g(X_2)=g$ and $X_2(\mathbb{R})$
has $s$ components $C_1^2, \cdots, C_s^2$ such that
$\delta_{C_i^2}(\pi_2)=\delta_i$ for $1\leq i\leq s'$ and
$\delta_{C_i^2}(\pi_2)=0$ for $s'+1\leq i\leq s$ hence $\pi _2$ has
topological degrees $\underline{\delta}$ and $X_2$ has topological
type $(g,s,1)$. In case $s\neq s'$ we apply Construction I
$k-2-\sum_{i=1}^{s'}\delta_i$ times on $C_{s'+1}^2$ using the
deformation with real ramification. In case $s=s'$ and $s'\neq 0$ we
apply Construction I $k-2-\sum_{i=1}^{s'}\delta_i$ times on $C_1^2$
alternating the deformation with and without real ramification. We
obtain a real curve $X$ of topological type $(g,s,1)$ and a real
morphism $\pi :X\rightarrow \mathbb{P}^1$ having degree $k$ and
topological degree $\underline{\delta}$. Indeed, in this final step
the degrees of the restrictions of $\pi$ do not change since
$k-2-\sum_{i=1}^{s'}\delta_i$ is even.

Now assume $s=0$ (the case of no real points). In this case $k$ is
even. In case $g<k$ a general divisor $D$ on a real curve $X$ of
degree $k$ gives rise to a base point free linear system $\mid
D\mid$. A general pencil in it defined over $\mathbb{R}$ gives rise
to a real base point free $g^1_k$ on $X$, hence to a real morphism
$\pi :X\rightarrow \mathbb{P}^1$ of degree $k$. So we can assume
$g\geq k$. Start with a hyperelliptic curve $\widetilde{X}$ of genus
$g-(k/2)+1$ with $\widetilde{X}(\mathbb{R})=\emptyset $. We apply
Construction IV $(k/2)-1$ times. Again we obtain a real curve $X$ of
genus $g$ with $X(\mathbb{R})=\emptyset$ and having a real morphism
$\pi :X\rightarrow \mathbb{P}^1$ of degree $k$.

\subsection{The case $a=0$}\label{subsection2.2}

Again, let $s'$ be the number of $\delta _i\neq 0$.

\subsubsection{Case 1}

Assume $s'=s=1$ and $\delta_1=k$. In particular $g\equiv 0$ (mod 2).
Let $X_0$ be a real hyperelliptic curve of genus $g$ having
$a(X_0)=0$, $s(X_0)=1$, hence for the hyperelliptic real morphism
$\pi_0 :X_0\rightarrow \mathbb{P}^1$ the unique component $C_1$ of
$X_0(\mathbb{R})$ has topological degree 2 (see \cite{ref2}*{Section
6}). Apply $k-2$ times Construction I using the deformation without
real ramification. We obtain a real curve $X$ of topological type
$(g,1,0)$ such that $X$ has a covering $\pi : X\rightarrow
\mathbb{P}^1$ defined over $\mathbb{R}$ of degree $k$ such that
$\delta_{C_1}(\pi)=k$.

\subsubsection{Case 2}

Assume $s>1$ and $\sum_{i=1}^{s}\delta_i=k$ in which case $s'=s$. In
case all $\delta_i=1$ then $s=k$ and therefore $g\equiv k+1$ (mod
2). In particular $g-k+2$ is odd. Since $k=s$ and $s\leq g+1$ also
$g-k+2\geq 1$. Let $X_0$ be a real hyperelliptic curve of genus
$g_0=g-k+2$ having $a(X_0)=0$, $s(X_0)=2$ (see \cite{ref2}*{Section
6}). Hence, for the real hyperelliptic morphism $\pi_0 :
X_0\rightarrow \mathbb{P}^1$ the 2 components $C$ of
$X_0(\mathbb{R})$ satisfy $\delta_C(\pi_0)=1$. Apply $k-2$ times
Construction III. We obtain a real curve $X$ of genus $g$ having
$a(X)=0$, $s(X)=k$ such that $X$ has a covering $\pi :X\rightarrow
\mathbb{P}^1$ defined over $\mathbb{R}$ of degree $k$ such that each
component $C$ of $X(\mathbb{R})$ satisfies $\delta_C(\pi)=1$.

Assume not all $\delta_i=1$, hence $\delta_1\geq 2$. Let $X_0$ be a
real hyperelliptic curve of genus $g_0=g-s+1$ (in particular $g_0$
is even) as in Case 1. Apply $s-1$ times Construction III. We obtain
a real curve $X_1$ of genus $g$ having $a(X)=0$, $s(X)=s$ such that
$X_1$ has a covering $\pi_1 :X_1\rightarrow \mathbb{P}^1$ defined
over $\mathbb{R}$ of degree $s+1$ such that exactly one component
$C_1$ of $X_1(\mathbb{R})$ satisfies $\delta_{C_1}(\pi_1)=2$ while
the other components $C_2, \cdots, C_s$ satisfy
$\delta_{C_i}(\pi_1)=1$. We apply $\delta_1-2$ times Construction I
without real ramification starting with $C_1$ and for $2\leq i\leq
s$ we apply $\delta_i-1$ times Construction I without real
ramification starting with $C_i$. We obtain a real curve $X$ of
genus $g$ satisfying $a(X)=0$, $s(X)=s$ such that $X$ has a covering
$\pi : X\rightarrow \mathbb{P}^1$ defined over $\mathbb{R}$ of
degree
$s+1+(\delta_1-2)+\sum_{i=2}^s(\delta_i-1)=\sum_{i=1}^s\delta_i=k$
having topological degree $\underline{\delta}$.

\subsubsection{Case 3}

Assume $\sum_{i=1}^{s'}\delta_i<k$ and $s'=s$. It follows that
$s'\geq 1$ and $g\equiv s'+1$ (mod 2). In particular $g-s'+1$ is
even. Let $X_0$ be a real hyperelliptic curve of genus $g-s'+1$ such
that $a(X_0)=0$, $s(X_0)=1$ and for the real hyperelliptic morphism
$\pi_0 : X_0\rightarrow \mathbb{P}^1$ one has that the unique
component $C_1$ of $X_0(\mathbb{R})$ satisfies
$\delta_{C_1}(\pi_0)=2$ (see \cite{ref2}*{Section 6}). Apply
Construction I using the deformation with real ramification to $C_1$
and apply $s'-1$ times Construction III. We obtain a real curve
$X_1$ of genus $g$ having $a(X_1)=0$, $s(X_1)=s'=s$ such that there
is a morphism $\pi_1 : X_1\rightarrow \mathbb{P}^1$ of degree $s'+2$
defined over $\mathbb{R}$ such that each component $C_i$ of
$X_1(\mathbb{R})$ satisfies $\delta_{C_i}(\pi_1)=1$.

For $1\leq i\leq s'$ apply $\delta_i-1$ times Construction I using
the deformation without real ramification starting with $C_i$. We
obtain a real curve $X_2$ of genus $g$ having $a(X_2)=0$,
$s(X_2)=s'=s$ and such that there is a morphism $\pi_2 :
X_2\rightarrow \mathbb{P}^1$ of degree
$s'+2+(\sum_{i=1}^s\delta_i)-s'=(\sum_{i=1}^s\delta_i)+2$ of
topological degree $\underline{\delta}$.

By assumption $k-(\sum_{i=1}^s\delta_i)-2$ is an even non-negative
integer. Apply $k-(\sum_{i=1}^s\delta_i)-2$ times Construction I
starting with $C_1$ alternating deformation with and without real
ramification (starting with real ramification) we end up with a real
curve $X$ of genus $g$ satisfying $a(X)=0$, $s(X)=s'=s$ and such
that there is a morphism $\pi :X\rightarrow \mathbb{P}^1$ of degree
$k$ of topological degree $\underline{\delta}$.

\subsubsection{Case 4}

Assume $\sum_{i=1}^{s'}\delta_i<k$, $s>s'$ and $s'\geq 1$. We start
using the construction of Case 3 making a real curve $X'$ of genus
$g-s+s'$ satisfying $a(X')=0$, $s(X')=s'$ and such that there is a
real morphism $\pi' :X'\rightarrow \mathbb{P}^1$ of degree $k$ such
that, for the components $C_1, \cdots, C_{s'}$ one has
$\delta_{C_i}(\pi')=\delta_i$. Notice that, since $g-s\equiv 1$ (mod
2) one has $g-s+s'\equiv s'+1$ (mod 2). Since
$\sum_{i=1}^{s'}\delta_i<k$ the set
$\pi^{-1}(\mathbb{P}^1(\mathbb{R}))\setminus X'(\mathbb{R})$ is not
empty. Therefore we can apply Construction II $s-s'$ times using a
deformation with real ramification. We obtain a real curve $X$ of
genus $g$ such that $a(X)=0$, $s(X)=s$ and $X$ has a real covering
$\pi :X\rightarrow \mathbb{P}^1$ of degree $k$ of topological degree
$\underline{\delta}$.

\subsubsection{Case 5}

Assume $s'=0$. In this case $k$ is even. Let $X_0$ be a
hyperelliptic curve of genus $g-s+1$ (hence of even genus) such that
$a(X_0)=0$, $s(X_0)=1$ and if $\pi_0 : X_0\rightarrow \mathbb{P}^1$
is the hyperelliptic morphism then the unique component $C_1$ of
$X_0(\mathbb{R})$ satisfies $\delta_{C_1}(\pi_0)=2$ (see
\cite{ref2}*{Section 6}). Apply $k-2$ times Construction I using the
deformation with real ramification. One obtains a real curve $X_1$
of genus $g-s+1$ such that $a(X_1)=0$ and $s(X_1)=1$ and a morphism
$\pi_1 : X_1\rightarrow \mathbb{P}^1$ of degree $k$ such that for
the unique component $C_1$ of $X_1(\mathbb{R})$ one has
$\delta_{C_1}(\pi_1)=0$. It follows that
$(\pi_1)^{-1}(\mathbb{P}^1(\mathbb{R}))\setminus X_1(\mathbb{R})\neq
\emptyset$ and so we apply $s-1$ times Construction II with real
ramification. We obtain a real curve $X$ of genus $g$ such that
$a(X)=0$, $s(X)=s$ and there exists a real morphism $\pi
:X\rightarrow \mathbb{P}^1$ of degree $k$ such that for each
component $C$ of $X(\mathbb{R})$ one has $\delta_C(\pi)=0$.

\subsection{Morphisms to $R_0$}\label{subsection2.3}

Assume $X$ is a real curve of genus $g$ with
$X(\mathbb{R})=\emptyset$ and assume $\pi :X \rightarrow R_0$ is a
morphism of degree $k$ defined over $\mathbb{R}$. Consider the
associated morphism of Riemann surfaces $\pi
(\mathbb{C}):X(\mathbb{C}) \rightarrow
R_0(\mathbb{C})=\mathbb{P}^1(\mathbb{C})$ and let $|D|$ be the
associated complete linear system on $X(\mathbb{C})$. We are going
to show that $|D|$ is not induced by a linear system on $X$ (i.e.
$|D|$ does not contain a real divisor).

Notice that complex conjugation of divisors on $X(\mathbb{C})$ acts
on $|D|$. Indeed, choose a non-real point $Q+\overline{Q}$ on $R_0$
and let $\pi (\mathbb{C})^{-1}(Q)=D_1$. Since $\pi
(\mathbb{C})^{-1}(\overline{Q})=\overline{D_1}$ one has $D_1 \sim
\overline{D_1}$. For any $E\in |D|$ one has $E\sim D_1$. This
implies $\overline{E}\sim \overline{D_1}$ hence also $E\sim
\overline{E}$. Now $\pi (\mathbb{C})$ corresponds to a line $L$ in
$|D|$ that is invariant under this complex conjugation. If $|D|$
would be real it would imply $L$ is a real line, hence containing a
real point. Such real point corresponds to a real divisor and this
would be a fiber of $\pi(\mathbb{C})$ invariant under complex
conjugation but such fibers do not exist.

This implies $\mathcal{O}_X (D)$ corresponds to a point of $\Pic ^k
(X)(\mathbb{R})\setminus \Pic ^k(X)(\mathbb{R})^+$, hence $k \equiv
g+1 \pmod{2}$ (see \cite{ref2}). In this section we consider pencils
corresponding to points in $\Pic ^k (X)(\mathbb{R})\setminus \Pic
^k(X)(\mathbb{R})^+$ (in particular $k \equiv g+1 \pmod{2}$).

\begin{theorem}\label{theorem3}
Let $g$ and $k \equiv g+1 \pmod{2}$ be nonnegative integers. In case
$k \geq g+1$ and $X$ is a real curve of  genus $g$ with
$X(\mathbb{R})=\emptyset$ then there exists a morphism $\pi : X
\rightarrow R_0$ of degree $k$. In case $k<g+1$ then there exists a
real curve of genus $g$ with $X(\mathbb{R})=\emptyset$ having a
morphism $\pi : X \rightarrow R_0$ of degree $k$.
\end{theorem}

\begin{proof}

First assume $k\geq g+1$. Choose an effective divisor $D$
 such that $\mathcal{O}_X (D)$ corresponds to a point on $\Pic ^k(X)(\mathbb{R})\setminus \Pic ^k(X)(\mathbb{R})^+$.
Then $|D|$ is not defined over $\mathbb{R}$ but complex conjugation
of divisors is defined on $|D|$. Choose $D_1\in |D|$ and let $L$ be
the line in $|D|$ connecting $D_1$ to $\overline{D_1}$. This
corresponds to a morphism $\pi (\mathbb{C}):X(\mathbb{C})\rightarrow
\mathbb{P}^1(\mathbb{C})$. Since $L=\overline{L}\subset |D|$ it
follows $\pi (\mathbb{C})$ is equivariant using the complex
conjugation on $X(\mathbb{C})$ and some antiholomorphic involution
on $\mathbb{P}^1(\mathbb{C})$ without fixed points. This implies
$\pi (\mathbb{C})$ is induced by a morphism $\pi : X\rightarrow R_0$
defined over $\mathbb{R}$.

Next assume $k\leq g+1$ and $k \equiv g+1 \pmod{2}$. Let $X_1$ be a
real curve of genus $g-k+2$ having a morphism $f_1:X_1\rightarrow
R_0$ of degree 2 (such curves exist since $g-k+2$ is odd, see
\cite{ref2}*{Section 6}). Then use $k-2$ times Construction V.

\end{proof}

\begin{remark}\label{remark4}
In his paper \cite{ref17} the author obtains Theorem \ref{theorem3}
by fixing a real divisor $D$ on $R_0$ and using the Riemann
Existence Theorem using coverings on $R_0(\mathbb{C})\setminus D$.
Although the author gives no argument for the existence of an
antiholomorphic involution on a suited choice of such a covering
this reasoning can be applied using Klein surfaces (see
\cite{AllGre} for the foundations of that theory). Also in that
paper he obtains Theorem \ref{theorem2} for the case $s=0$ but in
that argument details are omitted.

In his paper \cite{ref18} the author also gives an argument for the
existence of real curves $X$ of given topological type having a
morphism $\pi :X \rightarrow \mathbb{P}^1$ of degree $k$. For the
case $a(X)=0$ he uses deformations of hyperelliptic and trigonal
curves. The case $a(X)=1$ is contained in \cite{ref23}. In that
paper the author uses deformations of singular real curves on suited
real rational surfaces using Brusotti-type arguments. We use
arguments directly applied to coverings which is more natural and
easier. In both cases it is not clear to us how to obtain the
existence of all admissible topological degrees for $\pi$ using
those deformations. We also give some remarks on some of the
arguments used in those papers (especially \cite{ref18}) in Remark
\ref{remark9}.

In \cite{ref2}*{Section 8} one considers real trigonal curves and
one asks about determining the topological types of such curves. In
Theorem \ref{theorem2} we obtain that all admissable topological
types can be realized as real trigonal curves. A (in principle
similar) proof is described in \cite{ref18}*{Proposition 2.1}. A
different proof using Fuchsian and NEC (non-euclidean
crystallographic) groups (and obtaining much more detailled
information concerning the case of cyclic trigonal coverings) is
given in \cite{ref22}.
\end{remark}

\section{Some Brill-Noether Theory for real pencils on real curves}

On $\Pic^k(X)(\mathbb{C})$ we consider the closed subspace
$W^1_k(X)$ representing invertible sheafs on $X_{\mathbb{C}}$
satisfying $h^0(L)\geq 2$. This subspace $W^1_k(X)$ is invariant
under conjugation, hence it is defined over $\mathbb{R}$. We write
$W^1_k(X)(\mathbb{R})$ to denote its set of real points and
$W^1_k(X)(\mathbb{R})^+=W^1_k(X)(\mathbb{R})\cap
\Pic^k(X)(\mathbb{R})^+$. In case $k\geq g+1$ one has $W^1_k(X)=\Pic
^k(X)(\mathbb{C})$, so we assume from now on that $k\leq g$.

The study of those spaces $W^1_k(X)$ for general complex curves $X$
of genus $g$ is the Brill-Noether Theory for pencils. The main
results are the following. First consider the Brill-Noether number
$\rho ^1_k(g)=2k-g-2$. For a general complex curve $X$ one has
$W^1_k(X)=\emptyset$ in case $\rho ^1_k(g)<0$ and $\dim
(W^1_k(X))=\rho ^1_k(g)$ in case $\rho ^1_k(g)\geq 0$. Also in case
$\rho ^1_k(g)\geq 0$ then $W^1_k(X)\neq \emptyset$ for all complex
curves $X$. Here ''general'' means: the statement holds for all
curves $X$ on a dense open subset of the moduli space $M_g$ of
curves of genus $g$.

In \cite{ref3} one considers Brill-Noether theory of
$W^1_k(X)(\mathbb{R})$ for real curves $X$. There are important
differences with the complex case. We consider the moduli space
$M_{g/\mathbb{R}}$ of real curves of genus $g$ representing the
isomorphism classes over $\mathbb{R}$ of real curves of genus $g$.
We recall some facts on it (for details see \cite{ref1}). It is a
semi-analytic real variety of dimension $3g-3$ (as a matter of fact
it has a non-empty boundary, see \cite{ref12} for its description).
For each topological type $(g,s,a)$ there is a unique connected
component $M_{g/\mathbb{R}}(g,s,a)$ of $M_{g/\mathbb{R}}$. This
connected component is the quotient of a Teichmuller space (which is
a connected real analytic manifold) by means of a discontinuous
action of a modular group. Inside such component
$M_{g/\mathbb{R}}(g,s,a)$, curves having a certain type of real
linear system give rise to semi-analytic subvarieties of
$M_{g/\mathbb{R}}(g,s,a)$. Therefore there is in general no
description of behavior on a dense open subset of
$M_{g/\mathbb{R}}(g,s,a)$ and we need to modify the meaning of
''general''.

\begin{definition}\label{definition5}
Let $P$ be some property concerning real linear systems on real
curves. We say there is \emph{a general real curve} of topological
type $(g,s,a)$ satisfying property P if there exists a non-empty
open subset $U$ of $M_{g/\mathbb{R}}(g,s,a)$ such that P holds for
all curves corresponding to a point of $U$.
\end{definition}

\begin{example}\label{exampleN}
To illustrate the difference with the complex situation we consider
$W^1_3$ for curves of genus 4 ($\rho ^1_3(4)=0\geq 0$) (see
\cite{ref2}*{Section 8}). A real curve of type (4,0,1) has no real
$g^1_3$ since all its real divisors have even degree. In case $s\neq
0$ there exists a general real curve of type (4,s,a) having a real
$g^1_3$. In case $(s,a)\neq (1,0)$ there also exists a general real
curve of type $(4,s,a)$ having no real $g^1_3$. In Proposition
\ref{proposition10}, as an application of Theorem \ref{theorem2}, we
give a description of the situation in case $(s,a)=(1,0)$.
\end{example}

In \cite{ref3} it is proved that in case $\rho ^1_k(g)\geq 0$ and
$s\neq 0$ then there is a general real curve of type $(g,s,a)$ such
that $W^1_k(X)(\mathbb{R})\neq \emptyset$. In order to prove this
result the author constructs a real compactification of the real
Hurwitz space of coverings of degree $k$. The restriction $s\neq 0$
has a technical cause. Working with families of real curves the
author needs the existence of a section defined over $\mathbb{R}$.
In particular such a section guaranties that the relative Picard
scheme of such families represents the Picard functor (see
\cite{ref11b}).

Now we are going to use the existence results from Section
\ref{section2} to give another proof. In that way we obtain a finer
statement involving the covering degree. In the argument we make
intensive use of Brill-Noether theory for complex curves and we work
with the space of morphisms to $\mathbb{P}^1$ instead of the Hurwitz
space. This space of morphisms is defined using Hilbert schemes in
\cite{ref11a}*{Section 4.c}. In particular we do not need relative
Picard schemes to finish our arguments, hence we do not need the
restriction $s\neq 0$. In case $s=0$ we consider both
$W^1_k(X)(\mathbb{R})^+$ and $W^1_k(X)(\mathbb{R})\setminus
W^1_k(X)(\mathbb{R})^+$.

In case $f:X \rightarrow R_0$ is a morphism of degree $k$ then
$f_{\mathbb{C}} : X_{\mathbb{C}} \rightarrow
\mathbb{P}^1_{\mathbb{C}}$ corresponds to a linear system $g^1_k
\subset |L|$ for some $L\in W^1_k(X)(\mathbb{R})\setminus
W^1_k(X)(\mathbb{R})^+$. In order to study
$W^1_k(X)(\mathbb{R})\setminus W^1_k(X)(\mathbb{R})^+$ using
morphims, we also need the converse statement.

\begin{lemma}\label{lemma6}
Let $L\in W^1_k(X)(\mathbb{R})\setminus W^1_k(X)(\mathbb{R})^+$ and
assume $|L|$ is base point free of dimension 1. Then there is a
morphism $f:X \rightarrow R_0$ of degree $k$ such that
$f_{\mathbb{C}}$ corresponds to $|L|$.
\end{lemma}

\begin{proof}
See \cite{ref4}*{Example 1}.
\end{proof}

Let $X$ be a real curve of topological type $(g,s,a)$ and let
$g^1_k$ be a complete base point free linear system on $X$. We say
$g^1_k$ has topological degree $\underline{\delta}$ if the morphisms
$f:X \rightarrow \mathbb{P}^1$ associated to $g^1_k$ have
topological degree $\underline{\delta}$.

Let ${ }^{\circ}{W^1_k}(X)$ be the subspace of $W^1_k(X)$
parameterizing complete base point free linear systems $g^1_k$ on
$X_{\mathbb{C}}$. It is the complement of $W^1_{k-1}(X)+W^0_1(X)$ in
$W^1_k(X)$, hence it is Zariski-open in $W^1_k(X)$. In particular ${
}^{\circ}{W^1_k}(X)(\mathbb{R})$ is Zariski-open in
$W^1_k(X)(\mathbb{R})$. Because the topological degree is a discrete
invariant it is constant on connected components of ${
}^{\circ}{W^1_k}(X)(\mathbb{R})$. Let ${
}^{\circ}{W^1_k}(X)(\underline{\delta})$ be the sublocus of ${
}^{\circ}{W^1_k}(X)(\mathbb{R})$ such that it parameterizes linear
systems $g^1_k$ of topological degree $\underline{\delta}$ and let
$W^1_k(X)(\underline{\delta})$ be its closure in
$W^1_k(X)(\mathbb{R})$ (it is a union of irreducible components of
$W^1_k(X)(\mathbb{R})$).

\begin{theorem}\label{theorem7}
Let $\underline{\delta}$ be an admissible topological degree of base
point free linear systems $g^1_k$ on real curves of topological type
$(g,s,a)$ such that $\sum_{i=1}^s \delta _i \leq k-2$ in case $a=1$.

\begin{enumerate}

\item If $\rho ^1_k(g)<0$ then there is no general real curve $X$ of
type $(g,s,a)$ such that $W^1_k(X)(\underline{\delta})$ is not
empty.

\item If $\rho ^1_k(g)\geq 0$ then there is a general real curve $X$
of topological type $(g,s,a)$ such that
$W^1_k(X)(\underline{\delta})$ is non-empty and it is a real
algebraic subset of $\Pic ^k(X)(\mathbb{R})$ of dimension $\rho
^1_k(g)$.

\item In case $s=0$ and $k \equiv g+1 \pmod {2}$ then in case $\rho
^1_k(g)<0$ there is no general real curve $X$ of topological type
$(g,0,1)$ such that $W^1_k(X)(\mathbb{R})\setminus
W^1_k(X)(\mathbb{R})^+$ is not empty. In case $\rho ^1_k(g)\geq 0$
then there is a general real curve $X$ of topological type $(g,0,1)$
such that $W^1_k(X)(\mathbb{R}) \setminus W^1_k(X)(\mathbb{R})^+$ is
not empty and has dimension $\rho ^1_k(g)$.

\end{enumerate}

\end{theorem}

In order to prove this theorem we are going to use spaces
parameterizing morphisms to $\mathbb{P}^1$. First we introduce some
definition that will be useful for the proof.

\begin{definition}\label{definition8}
A \emph{suited family of curves of genus $g$} is a projective
mophism $\pi :\mathcal{C}\rightarrow S$ such that
\begin{itemize}
\item $S$ is smooth, irreducible and quasi-projective,
\item each fiber of $\pi$ is a smooth connected curve of genus $g$,
\item for each $s\in S$ the Kodaira-Spencer map $T_s(S)\rightarrow
H^1(\pi ^{-1}(s), T_{\pi^{-1}(s)})$ is bijective.
\end{itemize}
In case $X$ is a given smooth curve of genus $g$ then we say $\pi
:\mathcal{C}\rightarrow S$ is a suited family for $X$ if moreover
some fiber of $\pi$ is isomorphic to $X$. If moreover $X$ is defined
over $\mathbb{R}$ then we also assume $\pi :\mathcal{C}\rightarrow
S$ is defined over $\mathbb{R}$ and $X$ is isomorphic over
$\mathbb{R}$ to some fiber of $\pi$ over a real point of $S$.
\end{definition}

Such suited families do exist (see \cite{ref4}*{Section 3}, in that
paper one considers real curves with real points and one needs the
family to have a real section, in our situation one does not need to
use $S'$ from loc. cit.). From a suited family we obtain a morphism
$S\rightarrow M_g$ and in case $K= \mathbb{R}$ we also obtain a
morphism $S(\mathbb{R})\rightarrow M_{g/\mathbb{R}}$. From the
bijectivity of the Kodaira-Spencer map it follows those morpisms to
the moduli space have finite fibers. Therefore it is enough to prove
the statements of Theorem \ref{theorem7} for real curves
corresponding to fibers of $\pi$ at general points of
$S(\mathbb{R})$.

Let $\pi : \mathcal{C}\rightarrow S$ be a suited family of curves of
genus $g$ defined over $K$ (being $\mathbb{C}$ or $\mathbb{R}$). For
a noetherian $S$-scheme $T$ we consider the set
$\mathcal{H}_k(\pi)(T)$ of finite $T$-morphisms $\mathcal{C}\times
_S T\rightarrow \mathbb{P}^1\times T$ of degree $k$. This functor is
representable by an $S$-scheme $\pi _k : H_k (\pi) \rightarrow S$
(see \cite{ref11a}*{Section 4.c}). Let $f: X \rightarrow
\mathbb{P}^1_{\mathbb{C}}$ be a morphism corresponding to a point
$[f]$ on $H_k(\pi)(\mathbb{C})$ and consider the associated exact
sequence

\[
0 \rightarrow T_X \rightarrow f^*(T_{\mathbb{P}^1}) \rightarrow N_f
\rightarrow 0
\]
From Horikawa's deformation theory of holomorphic maps (see
\cite{ref19}, see also \cite{ref13}*{3.4.2}) it follows
$T_{[f]}(H_k(\pi))$ is canonically identified with $H^0(X,N_f)$ and
since $H^1(X,N_f)=0$ it follows $H_k(\pi)$ is smooth of dimension
$2k+2g-2$. In case $K= \mathbb{R}$ it follows $H_k(\pi)(\mathbb{R})$
is a smooth real manifold of dimension $2k+2g-2$ (of course $H_k(\pi
)(\mathbb{R})$ need not be connected).

In case $K=\mathbb{R}$ then for a noetherian $S$-scheme $T$ we
consider the set $\mathcal{H}^R_k(\pi)(T)$ of finite $T$-morphisms
$\mathcal{C} \times _S T \rightarrow R_0 \times T$ of degree $k$.
This functor is represented by an $S$-scheme $\pi ^R_k :H^R_k(\pi)
\rightarrow S$ and using the same arguments we obtain $H^R_k(\pi)$
is smooth of dimension $2k+2g-2$ and $H^R_k(\pi)(\mathbb{R})$ is a
smooth real manifold of dimension $2k+2g-2$ (again, it need not be
connected).

Because $\dim (Aut(\mathbb{P}^1_{\mathbb{C}}))=3$ the non-empty
fibers of $\pi_k(\mathbb{C})$ (and $\pi ^R_k(\mathbb{C})$) do have
dimension at least 3. Hence the image of $\pi _k(\mathbb{C})$ (and
of $\pi ^R_k(\mathbb{C})$) has dimension at most $2k+2g-5=3g-3+\rho
^1_k(g)$. Hence in case $\rho ^1_k(g)<0$ then a general point of
$S(\mathbb{C})$ does not belong to the image of $\pi _k(\mathbb{C})$
(or $\pi ^R_k (\mathbb{C})$). Now assume $\rho ^1_k(g)\geq 0$. From
the arguments in \cite{ref14}*{Appendix} it follows that each
component of $H_k(\pi)(\mathbb{C})$ (or $H^R_k(\pi)(\mathbb{C})$)
dominates $S(\mathbb{C})$. Let $H'_k(\pi)(\mathbb{C})$ (resp.
$H'^R_k(\pi)(\mathbb{C})$) be the closed subspace of
$H_k(\pi)(\mathbb{C})$ (resp. $H^R_k(\pi)(\mathbb{C})$) consisting
of the fibers of $\pi_k(\mathbb{C})$ (resp. $\pi^R_k(\mathbb{C})$)
having dimension greater than $\rho ^1_k(g) +3$. Assume it contains
an irreducible component of $H_k(\pi )(\mathbb{C})$. Since such
component dominates $S(\mathbb{C})$ it would follow that it has
dimension more than $3g-3+\rho ^1_k(g)+3=2g+2k-2$, a contradiction.
Hence it follows that $H'_k(\pi )(\mathbb{C})$ (resp. $H'^R_k(\pi
)(\mathbb{C})$) does not contain irreducible components of
$H_k(\pi)(\mathbb{C})$ (resp. $H^R_k(\pi)(\mathbb{C})$), hence its
dimension is less than $2g+2k-2$. Moreover, those subspaces are
invariant under complex conjugation, hence they are defined over
$\mathbb{R}$. So we obtain closed subschemes $H'_k(\pi)$ and
$H'^R_k(\pi)$ of $H_k(\pi )$ defined over $\mathbb{R}$ and we obtain
$\dim \left( H'_k(\pi)(\mathbb{R}) \right) <2g+2k-2$ (resp. $\dim
\left( H'^R_k(\pi)(\mathbb{R}) \right) <2g+2k-2$). For a complex
curve $X_s$ corresponding to a general point $s$ on $S(\mathbb{C})$
it follows from Brill-Noether Theory that $W^r_k(X)$ has dimension
$\rho ^r_k(g)=g-(r+1)(g-k+r)$ (if $\rho ^r_k(g)\geq 0$, otherwise it
is empty). In case $r\geq 2$ this gives rise to a subspace of
$\pi^{-1}_k(s)$ (resp. $\pi '^{-1}_k (s)$) of dimension $\rho
^r_k(g)+2(r-1)+3<\rho ^1_k(g)+3$. Let $H^2_k (\pi)(\mathbb{C})$
(resp. $H^{2,R}_k(\pi)(\mathbb{C})$) be the closed subspace of
$H_k(\pi)(\mathbb{C})$ (resp. $H^R_k(\pi)(\mathbb{C})$)
corresponding to linear systems $g^1_k$ that are not complete. Then
we obtain that those are not components of $H_k(\pi)(\mathbb{C})$
(resp. $H^R_k(\pi)(\mathbb{C})$) because they have dimension less
than $2g+2k-2$. Again $H^2_k(\pi)$ and $H^{2,R}_k(\pi)$ are defined
over $\mathbb{R}$ and we find $\dim \left( H^2_k(\pi)(\mathbb{R})
\right)<2g+2k-2$ (resp. $\dim \left( H^{2,R}_k(\pi)(\mathbb{R})
\right) <2g+2k-2$).

\begin{proof}[Proof of Theorem \ref{theorem7}]

In case $\rho ^1_k(g)<0$ then from $\dim \left( H_k(\pi)(\mathbb{R})
\right) <3g-3$ part 1 follows. Also in case $s=0$ one has $\dim
\left( H^R_k(\pi )(\mathbb{R})\right) <3g-3$, hence using Lemma
\ref{lemma6} the first statement of part 3 follows. So assume $\rho
^1_k(g)\geq 0$ .

By Theorem \ref{theorem2} there is a real curve $X$ of topological
type $(g,s,a)$ having a morphism $f:X \rightarrow \mathbb{P}^1$ of
degree $k$ having topological degree $\underline{\delta}$. Let $\pi
: \mathcal{C} \rightarrow S$ be a suited family for $X$ then $f$
defines $[f]\in H_k(\pi)(\mathbb{R})$. From the previous
consideration it follows we can assume $[f]\notin
H'_k(\pi)(\mathbb{R})\cup H^2_k(\pi)(\mathbb{R})$. From $[f]\notin
H'_k(\pi)(\mathbb{R})$ it follows $\dim \left( \pi_k (H_k(\pi
)(\mathbb{R})) \right) =3g-3$ and from $[f]\notin
H^2_k(\pi)(\mathbb{R})$ it follows $f$ corresponds to a complete
$g^1_k$. Hence a general element of $\pi _k(H_k(\pi )(\mathbb{R}))$
corresponds to a real curve $X'$ of type $(g,s,a)$ with $\dim \left(
W^1_k(X')(\underline{\delta}) \right) \geq \rho ^1_k(g)$. Also $\dim
\left( H_k(\pi)(\mathbb{R}) \right)=2g+2k-2$ implies we obtain $\dim
\left( W^1_k (X')(\underline{\delta}) \right)=\rho ^1_k(g)$, proving
part 2.

In case $s=0$ and $k \equiv g+1 \pmod {2}$, by Theorem
\ref{theorem3} there exists a real curve $X$ of topological type
$(g,0,1)$ having a real morphism $f:X\rightarrow R_0$ of degree $k$.
Now we find $[f]\in H^R_k(\pi )(\mathbb{R})$ and we can assume
$[f]\notin H'^R_k(\pi )(\mathbb{R})\cup H^{2,R}_k(\pi
)(\mathbb{R})$. As in the previous case we obtain the second
statement of part 3 (again we also use Lemma \ref{lemma6}).

\end{proof}

\begin{remark}\label{remark9}
The proof of Theorem \ref{theorem7} is intensively based on known
dimension statements in the complex case. In order to be able to use
those statements it is very important to know that $\dim \left(
H_k(\pi ) \right)=2k+2g-2$ implies $\dim \left( H_k(\pi
)(\mathbb{R}) \right)=2k+2g-2$. To make this conclusion it is very
important to know that $H_k(\pi )(\mathbb{C})$ is smooth. In his
paper \cite{ref18} mentioned in Remark \ref{remark4} the author uses
similar arguments to obtain statements as those in Theorem
\ref{theorem7} but without considering the topological degree.
However the arguments are applied on the moduli space itself.
Although the author claims the locus of $k$-gonal curves on the
module space is smooth this is not clear (and in general it is not
true). Moreover for his arguments he needs that the closure of that
locus is smooth at the hyperelliptic or trigonal locus. To solve
this technical problem one should use a suited family containing a
hyperelliptic or trigonal curve and use the space parameterizing
linear systems $g^1_k$ on fibers of that family. Indeed, according
to \cite{ref5} that space is smooth. The construction of that space
uses the relative Picard scheme and we did not use that space in our
arguments in order to be able to use arguments also applicable in
case $s=0$. Also in his paper \cite{ref23}*{End of proof of Theorem
0.1} the author uses similar arguments but again not all details are
described completely. In particular the author uses an argument on
Hurwitz schemes but it is not clear from it that one obtains real
coverings with only simple ramification (in particular, that one
obtains a point in the smooth locus of the compactified Hurwitz
scheme). Those arguments in \cite{ref23} can be replaced by using
the Hilbert scheme parameterizing morphisms and the details
described in this section.

In his paper \cite{ref18} the author also shows that (with the same
remarks on the proof as above), in case $\rho ^1_k(g)<0$, a general
$k$-gonal real curve of topological type $(g,s,a)$ has a unique
linear system $g^1_k$. This fact can be obtained from our arguments
(and fixing the topological degree of the linear system). As in
\cite{ref18} one uses that a general complex $k$-gonal curve has a
unique $g^1_k$ (see \cite{ref7}). Since a general complex $k$-gonal
curve has no multiple $g^1_k$ (see \cite{ref6}) one also obtains a
similar result for general real $k$-gonal curves fixing the
topological degrees.
\end{remark}

We return to Example \ref{exampleN}, finishing the case of curves of
topological type (4,1,0). A non-hyperelliptic complex curve $X$ of
genus 4 is trigonal. By Riemann-Roch, in case $g$ is a $g^1_3$ on
$X$ then $|K_X-g|=h$ is also a $g^1_3$ on $X$ and it is well known
that $X$ has no other linear systems $g^1_3$. Hence $X$ has either
two (in case $g\neq h$) or one (in case $g=h$) linear systems
$g^1_3$.

\begin{proposition}\label{proposition10}
Let $X$ be a non-hyperelliptic real curve of topological type
$(4,1,0)$. Then $X$ has two real linear systems $g^1_3$. One of them
has topological degree $(3)$, the other one has topological degree
$(1)$.
\end{proposition}

\begin{proof}
Let $T$ be a real Teichm\"uller space of real curves of topological
type (4,1,0) and let $\pi _T:\mathcal{X}\rightarrow T$ be the
associated universal family. This space $T$ is a connected real
manifold of dimension 9 and all real curves of topological type
(4,1,0) do occur as fibers of that family. Inside $T$ the closed
subspace $H$ of points having hyperelliptic fibers for $\pi _H$ has
codimension 2, hence $T \setminus H$ is still connected. Inside $T
\setminus H$ we consider $T(1)$ (resp. $T(3)$) being the subspace of
points such that the fiber $X$ for $\pi _H$ has a morphism $f:X
\rightarrow \mathbb{P}^1$ of degree 3 of topological degree (1)
(resp. (3)). From \cite{ref20} we know $T(3)=T \setminus H$ and we
need to prove $T(1)=T \setminus H$ too. Because of Theorem
\ref{theorem2} we know $T(1)\neq \emptyset$, we assume that
$T(1)\neq T\setminus H$.

The fiber of a point belonging to the closure of $T(1)$ in
$T\setminus H$ is a real curve $X$ that is the limit of real curves
$X_t$ having a morphism $f_t:X_t \rightarrow \mathbb{P}^1$ of
topological degree (1). This corresponds to a linear system
$g^1_3(t)$ on $X_t$ and the limit is a linear system $g^1_3$ on $X$.
Since $X$ is not hyperelliptic this linear system has no base points
and it is complete, hence it corresponds to a morphism $f:X
\rightarrow \mathbb{P}^1$ being a limit of those morphisms $f_t$,
hence $f$ has topological degree (1) too. This implies $T(1)$ is
closed in $T \setminus H$ hence it is not open in $T \setminus H$
since $T \setminus H$ is connected. Therefore $T(1)$ has a boundary
point $t_0$ in $T \setminus H$, let $X_0$ be the fiber of $\pi _T$
above $t_0$. Now let $\pi : \mathcal{C} \rightarrow S$ be a  suited
morphism of curves for $X_0$ (let $s_0\in S(\mathbb{R})$ such that
$\pi ^{-1}(s_0)=X_0$).

Let $\pi _3 : H_3 \rightarrow S$ be the parameterspace for all
coverings of degree 3 to $\mathbb{P}^1$ of fibers of $\pi$ above
points on $S$. Let $H_3(1)$ (resp. $H_3(3)$) be the closed open
subsets of coverings of topological degree $(1)$ (resp. $(3)$). For
a covering $f:X\rightarrow \mathbb{P}^1$ we write $L_f$ to denote
the corresponding invertible sheaf. For $[f]\in H_3$ it follows from
the deformation theory of Horikawa that $d_{[f]}(\pi _3)$ is
surjective if and only if $\dim \left(H^0(L_f^{\otimes 2})
\right)=3$ (indeed $H^0(X,N_f)\rightarrow H^1(X,T_X)$ is the tangent
map of $\pi _3$ at $[f]$). In that case for $x=\pi _3([f])$ there is
a neighborhood $U$ of $x$ in $S$ such that $U\subset \im (\pi _3)$.
In case $\dim \left( H^0(L_f^{\otimes 2}) \right)>3$ then $L_f
^{\otimes 2} \cong \omega _{\pi ^{-1}(x)}$, hence $L_f$ is
half-canonical.

Let $f_0 : X_0 \rightarrow \mathbb{P}^1$ be the morphism of degree 3
on $X_0$ of topological degree (1). For each classical neighborhood
$U$ of $s_0$ in $S$ there exists $s\in U\cap S(\mathbb{R})$ such
that $\pi ^{-1}(s)=\pi _T^{-1}(t)$ for some $t\notin T(1)\cup H$,
hence $U \nsubseteq \pi _3(H_3(1))$. From the previous description
of the tangent map of $\pi _3$ it follows $L_{f_0}$ is
half-canonical on $X_0$. In particular $X_0$ has no $g^1_3$
associated to an invertible sheaf different from $L_{f_0}$ and this
would imply there is no morphism $f:X_0 \rightarrow \mathbb{P}^1$ of
topological degree (3), hence $t_0\notin T(3)$ contradicting $T(3)=T
\setminus H$.

\end{proof}

\begin{remark}\label{remark11}
In contrast with our Theorem \ref{theorem7} it is already mentioned
that for each topological type $(4,s,a)\neq (4,1,0)$ there is a
general real curve $X$ having no real $g^1_3$. In \cite{ref3} it is
also proved that for each topological type $(8,s,a)\neq \{ (8,1,0),
(8,0,1) \}$ there is a general real curve $X$ having no real $g^1_5$
(while Theorem \ref{theorem7} implies there is a general real curve
$X$ having a real $g^1_5$). On the other hand, the main result of
\cite{ref20} implies that there is no general real curve $X$ of
topological type $(8,1,0)$ having no base point free $g^1_5$ of
topological degree $(5)$. For real curves without real points it is
proved in \cite{ref21} that for each genus $g$ there exist general
real curves $X$ of topological type $(g,0,1)$ such that
$X^1_k(X)(\mathbb{R})\setminus W^1_k(X)(\mathbb{R})^+$ is empty for
each $k\leq g$. Our Theorem \ref{theorem7} implies that there exist
general real curves of topological type $(g,g-1,0)$ having a base
point free $g^1_{g-1}$ of topological degree $(1, \cdots, 1)$. From
forthcoming work of the first author it follows that there also
exist general real curves $X$ of topological type $(g,g-1,0)$ having
no such linear system $g^1_{g-1}$ and many similar statements.
\end{remark}

\section{Real 4-gonal curves having a $g^1_4$ with no totally
non-real divisor and no component of non-zero degree}

\begin{definition}\label{defintion12}
Let $X$ be a real curve and let $f:X\rightarrow \mathbb{P}^1$ be a
morphism defined over $\mathbb{R}$. We define the \emph{covering
number} $k(f)$ as follows. Consider the associated map
$f(\mathbb{R}):X(\mathbb{R})\rightarrow \mathbb{P}^1(\mathbb{R})$.
If $f(\mathbb{R})$ is not surjective then $k(f)=0$. Otherwise $k(f)$
is the minimal number $k$ such that there exist connected components
$C_1, \cdots, C_k$ of $X(\mathbb{R})$ such that $f(C_1 \cup \cdots
\cup C_k)=\mathbb{P}^1(\mathbb{R})$.
\end{definition}

Of course, in case $X(\mathbb{R})=\emptyset$ then $k(f)=0$ by
definition. Also $k(f)=1$ if and only if there is a connected
component $C$ of $X(\mathbb{R})$ such that $\delta _C(f)\geq 1$. In
particular $k(f)=1$ in case $\deg(f)$ is odd. Also in case
$\deg(f)=2$ then $k(f)$ is equal to 0 or to 1. So the case with
$\deg (f)=4$ is the first interesting case to study the possible
values for $k(f)$. In the next theorem we prove there are no further
restrictions on $k(f)$ for gonality 4.

\begin{theorem}
Let $(g,s,a)$ be an admissable topological type for real curves with
$s\geq 1$. Let $1\leq k\leq s$. There exists a real curve $X$ of
topological type $(g,s,a)$ such that there is a covering
$f:X\rightarrow \mathbb{P}^1$ of degree 4 defined over $\mathbb{R}$
such that $k(f)=k$ and for each connected component $C$ of
$X(\mathbb{R})$ one has $\delta _C(f)=0$.
\end{theorem}

\begin{proof}
We first prove the case of $M$-curves ($s=g+1$, in particular $a=0$)
with $k=g+1$.

In case $g$ is even take two hyperelliptic $M$-curves $Y_1$ and
$Y_2$ of genus $g'=g/2$ having double coverings $f_i:Y_i\rightarrow
\mathbb{P}^1$ defined over $\mathbb{R}$ with the following
properties. Let $C^{i}_0, \cdots, C^{i}_g$ be the connected
components of $Y_i(\mathbb{R})$ and let $I^{i}_j=f_i(C^{i}_j)$ then
$I^1_{j_1}$ intersects $I^2_{j_2}$ if and only if one of the
following holds:

\begin{equation*}
\begin{cases}
j_1=j_2 & \text { with } 0\leq j_1 \leq g'\\
j_1=j_2+1 & \text { with } 1\leq j_1\leq g'\\
j_1=0 \text{ and } j_2=g'
\end{cases}
\end{equation*}
and the non-empty intersections are connected (see Figure
\ref{Figure 9} with $g'=2$).

\begin{figure}[h]
\begin{center}
\includegraphics[height=6 cm]{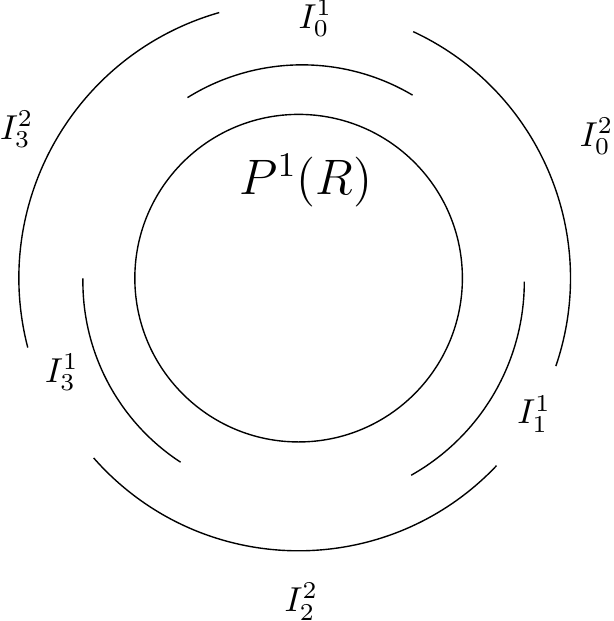}
\caption{$s=k=g+1$ and $g$ even} \label{Figure 9}
\end{center}
\end{figure}

\noindent Let $t\in I^1_0\cap I^2_0$ and take $p_i\in f^{-1}_i(t)$.
Let $X_0=\left( Y_0 \cup Y_1 \right) _{p_1=p_2}$. It is a stable
real curve of genus $g$ having a node $p=(p_1=p_2)$ and a morphism
$f:X_0\rightarrow \mathbb{P}^1$ of degree 4 defined over
$\mathbb{R}$. Locally at $p$ the curve $X_0$ is defined over
$\mathbb{R}$ by the equation $x^2-y^2=0$ and the morphism is given
by $(x,y) \mapsto x$. Using a local deformation over $\mathbb{R}$
given by $x^2-y^2=t$ and gluing with other local coordinates one
obtains a curve $X_t$ defined over $\mathbb{R}$ having a covering
$f_t:X_t\rightarrow \mathbb{P}^1$ of degree 4 defined over
$\mathbb{R}$ such that $g(X_t)=2g'=g$. (The details are as in
Construction 1.) The components of $X_t(\mathbb{R})$ are
deformations $C^{i}_j(t)$ of $C^{i}_j$ for $1\leq j\leq g'$ and
$i=1; 2$ and of $\left( C^1_0 \cup C^2_0 \right)_{p_1=p_2}$ which is
one component $C_0(t)$. Clearly $f_t\left( C_0(t) \right)$ is a
deformation of $I^1_0 \cup I^2_0$ and $f_t\left( C^{i}_j(t) \right)$
is a deformation of $I^{i}_j$ for $i=1; 2$, $1\leq j\leq g'$. It
follows that their union is equal to $\mathbb{P}^1(\mathbb{R})$ but
omitting one of them does not cover $\mathbb{P}^1(\mathbb{R})$. It
follows $s(X_t)=2g'+1=g+1$ (hence $X_t$ is an $M$-curve) and
$k=g+1$.

In case $g$ is odd take two hyperelliptic $M$-curves $Y_1$ of genus
$g'=(g-1)/2$ and $Y_2$ of genus $g'+1=(g+1)/2$. Use the notation
$C^{i}_j$ for the components of $Y_i(\mathbb{R})$ with $0\leq j\leq
g'$ in case $i=1$ and $0\leq j\leq g'+1$ in case $i=2$. Let
$I^{i}_j=f_i\left( C^{i}_j \right)$ and assume $I^1_{j_1}$
intersects $I^2_{j_2}$ for $0\leq j_1; j_2\leq g'$ as in the
previous case and $I^2_{g'+1}\subset I^1_0$ (see Figure \ref{Figure
10} with $g'=2$).

\begin{figure}[h]
\begin{center}
\includegraphics[height=6 cm]{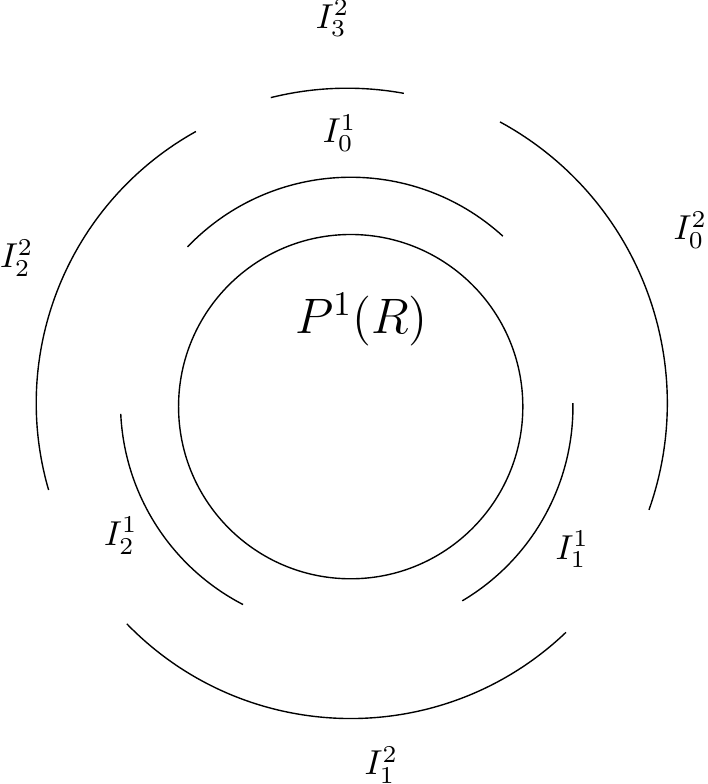}
\caption{$s=k=g+1$ and $g$ odd} \label{Figure 10}
\end{center}
\end{figure}

\noindent Let $t\in I^2_{g'+1}$ and take $p_i\in f^{-1}_i(t)$. Let
$X_0$ be as before then arguing as before one obtains an $M$-curve
$X_t$ of genus $g$ with a covering $f_t:X_t\rightarrow \mathbb{P}^1$
of degree 4 defined over $\mathbb{R}$ such that the components of
$X_t(\mathbb{R})$ are deformations $C^{i}_j(t)$ for $0\leq j\leq g'$
and $i=1; 2$ (with $j\neq 0$ for $i=1$) and $C(t)$ of $\left(
C^1_0\cup C^2_{g'+1}\right) _{p_1=p_2}$. The images $f_t \left(
C^{i}_j(t) \right)$ are deformations of $I^{i}_j$ and the image
$f_t(C(t))$ is a deformation of $I^1_0$. It follows $k=g+1$.

Now we are going to prove the theorem for $M$-curves in case
$k<g+1$. From the previous part of the proof we obtain the existence
of an $M$-curve $Y$ of genus $k-1$ having a covering
$f_Y:Y\rightarrow \mathbb{P}^1$ of degree 4 defined over
$\mathbb{R}$ satisfying the following property. Let $C_1, \cdots,
C_k$ be the components of $Y(\mathbb{R})$, let $I_i=f_Y(C_i)$ then
$I_i$ intersects $I_j$ for $i\neq j$ if and only if

\begin{equation*}
\begin{cases}
j=i+1 & \text { for } 1\leq i\leq k-1\\
j=k \text { and } i=1
\end{cases}
\end{equation*}
and the non-empty intersections are connected. In case $k=1$ then
$Y(\mathbb{R})$ has a unique component $C_1$ dominating
$\mathbb{P}^1(\mathbb{R})$ such that $\delta _{f_Y}(C_1)=0$ and
there is a connected closed subset $I\subset
\mathbb{P}^1(\mathbb{R})$ such that $x \in I$ if and only if
$f_y^{-1}(x) \subset Y(\mathbb{R})$.

Let $c_1, \cdots, c_{g-k+1}$ be different points on $I_1\cap I_2$ in
this order (with $c_1$ most close to $I_1\setminus I_2$) and let
$f_Y^{-1}(x_i)\cap C_1= \{ p_{i1}, p_{i2} \}$ (see Figure
\ref{Figure 11} with $k=4$ and $g=6$). In case $k=1$ those are
points in the inner part of $I$. Let $X_0=Y_{p_{i1}=p_{i2} \text{
for } 1\leq i\leq g-k+1}$. Then $X_0$ is defined over $\mathbb{R}$
and it has a covering $f_0:X_0\rightarrow \mathbb{P}^1$ defined over
$\mathbb{R}$ of degree 4. Locally at the node $p_i=\left(
p_{i1}+p_{i2} \right)$ the curve $X_0$ is defined over $\mathbb{R}$
by $x^2-y^2=0$ and the morphism by $(x,y)\mapsto x$. Using local
deformations over $\mathbb{R}$ by the equation $x^2-y^2=t$ with
$t\geq 0$ one obtains two new real ramification points close to
$c_i$. Gluing one obtains a curve $X_t$ defined over $\mathbb{R}$
having a morphism $f_t:X_t\rightarrow \mathbb{P}^1$ of degree 4
defined over $\mathbb{R}$ with $g(X_t)=g$. (This is similar to
applying $g-k+1$ times Construction I with real ramification.) For
each $2\leq i\leq k$ there is a component $C_i(t)$ of
$X_t(\mathbb{R})$ that is a deformation of $C_i$. Because of the
chosen local deformations, the deformation of $C_1$ is a union of
$g-k+2$ components $C_1(t), C'_1(t), \cdots, C'_{g-k+1}(t)$. It
follows $s(X_t)=g+1$ hence $X_t$ is an $M$-curve. The images
$f_t(C_i(t))$ are deformations of $I_i$ for $2\leq i\leq t$, the
image $f_t(C_1(t))$ is a deformation of the connected component of
$I_1\setminus \{ c_1 \}$ containing $I_1\setminus I_2$, the images
of $f_t(C'_i(t))$ are deformations of the interval between $c_i$ and
$c_{i+1}$ on $I_1$ for $1\leq i\leq g-k$ and a deformation of the
connected component of $I_1\setminus \{ c_{g-k+1} \}$ not containing
$I_1 \setminus I_2$ for $i=g-k+1$.

\begin{figure}[h]
\begin{center}
\includegraphics[height=6 cm]{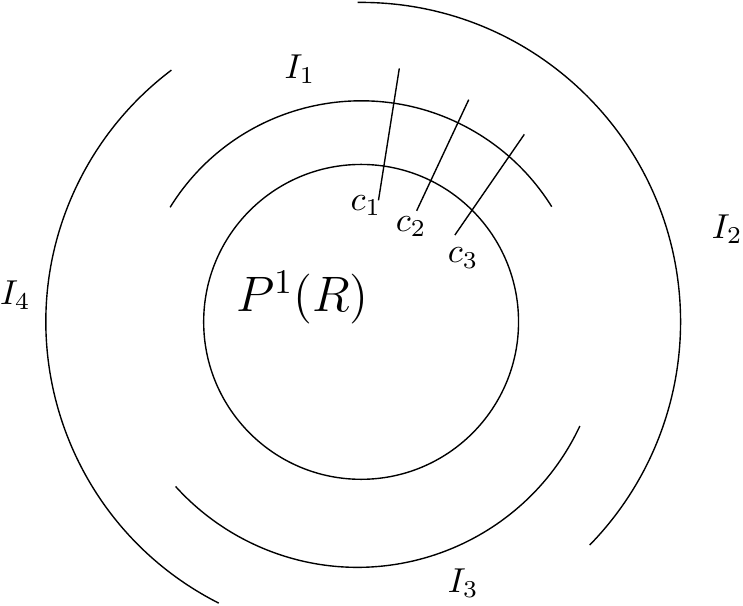}
\caption{$k<s=g+1$} \label{Figure 11}
\end{center}
\end{figure}

\noindent If follows $C_1(t), \cdots, C_k(t)$ is the only subset of
$k$ components $C$ of $X_t(\mathbb{R})$ such that the union of the
intervals $f_t(C)$ equals $\mathbb{P}^1(\mathbb{R})$. This implies
$k(f_t)=k$.

Now we finish the proof for the orientable case $a=0$ in case $X$ is
not an $M$-curve (in particular $s\leq g-1$). Let $b=(g+1-s)/2$,
which is an integer. From the prevous part of the proof we obtain
the existence of an $M$-curve $Y$ of genus $g-b=s+b-1$ (hence
$Y(\mathbb{R})$ has $s+b$ connected components) having a covering
$f_Y:Y\rightarrow \mathbb{P}^1$ of degree 4 defined over
$\mathbb{R}$ satisfying the following property. Let $C_1, \cdots,
C_{s+b}$ be the components of $Y(\mathbb{R})$ and let
$I_i=f_Y(C_i)$. Then $I_j$ and $I_{j'}$ with $1\leq j; j'\leq k+b$
do have a point in common if and only if

\begin{equation*}
\begin{cases}
j'=j+1 & \text { for } 1\leq j\leq k+b-1\\
j=k+b \text { and } j'=1
\end{cases}
\end{equation*}

\noindent and the non-empty intersections are connected and
$I_{k+b+1}, \cdots, I_{s+b}$ is contained in $I_1 \setminus \left(
I_2 \cup I_{k+b} \right)$ (see Figure \ref{Figure 12} with $k+b=4$
and $s+b=6$).

\begin{figure}[h]
\begin{center}
\includegraphics[height=6 cm]{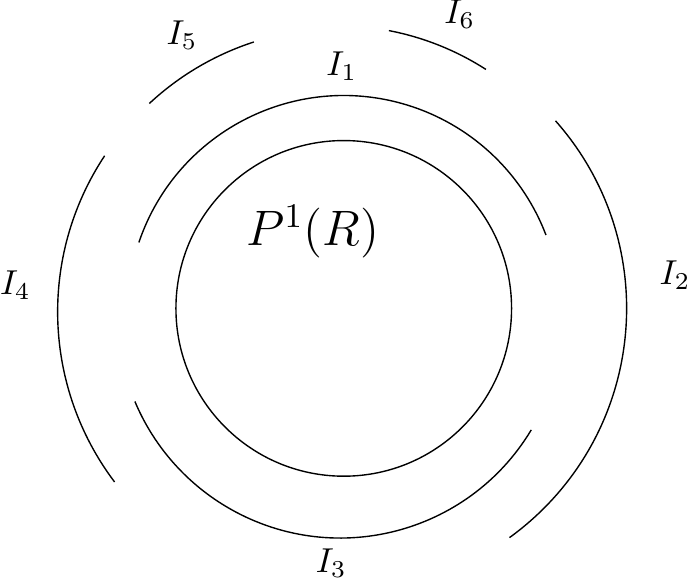}
\caption{$a=0$; $s<g+1$} \label{Figure 12}
\end{center}
\end{figure}

\noindent For $1\leq j\leq b$ let $p_j \in C_j$ and $q_j \in
C_{j+1}$ with $f_Y(p_j)=f_Y(q_j)\in I_j\cap I_{j+1}$. Let
$X_0=Y_{p_j=q_j \text { for } 1\leq j\leq b}$. Using locally real
deformations at the nodes $p_j=q_j$ one obtains a real curve $X_t$
of genus $g(Y)+b=g$ and a covering $f_t=X_t\rightarrow \mathbb{P}^1$
of degree 4 defined over $\mathbb{R}$. The components of
$X_t(\mathbb{R})$ are a deformation $C(t)$ of $\left( C_1 \cup
\cdots \cup C_{b+1} \right) _{p_j=q_j \text { for } 1\leq j\leq b}$
and deformations $C_i(t)$ of $C_i$ for $b+2\leq i\leq s+b$. In
particular $s(X_t)=s$. The image $f(C(t))$ is a deformation of $I_1
\cup \cdots \cup I_{b+1}$ and the image $f(C_i(t))$ for $b+2\leq
i\leq s+b$ are deformations of $I_i$. It follows that $k(f_t)=k$.

Finally we consider the non-orientable case (hence $a=1$). First
assume $s \equiv g \pmod { 2 }$. Let $Y$ be an orientable curve of
genus $g-1$ with $s(Y)=s$ such that $Y$ has a covering
$f_Y:Y\rightarrow \mathbb{P}^1$ of degree 4 defined over
$\mathbb{R}$ with $k(f_Y)=k$. Let $c\in \mathbb{P}^1(\mathbb{R})$
such that $f_Y^{-1}(c)$ is not a totally real divisor and let
$P+\overline{P}\subset f_Y^{-1}(c)$ for a non-real pair
$P+\overline{P}$. Let $X_0=Y_{P=\overline{P}}$, it is a real curve
of genus $g+1$ having a covering of degree 4 defined over
$\mathbb{R}$. Locally at the node $P=\overline{P}$ the curve is
defined over $\mathbb{R}$ by $x^2+y^2=0$ and the morphism by
$(x,y)\mapsto x$. Take a local deformation of the the type
$Z(x^2+y^2-t)$ with $t<0$ and glue it to obtain a real curve $X_t$
(this corresponds to Construction II without real ramification).
Clearly $X_t$ is non-orientable, $g(X_t)=g$, $s(X_t)=s$ and it has a
covering $f_t:X_t\rightarrow \mathbb{P}^1$ defined over $\mathbb{R}$
of degree 4 with $k(f_Y)=k$. Finally assume $s \equiv g+1 \pmod { 2
}$. Since $s\leq g-1$ we can take a non-orientable real curve $Y$ of
genus $g-1$ such that $s(Y)=s$ and there is a covering
$f_Y:Y\rightarrow \mathbb{P}^1$ of degree 4 defined over
$\mathbb{R}$ with $k(f_Y)=k$. Repeating the previous construction
one obtains the curve $X_t$ one is looking for.

\end{proof}

\vspace{2\baselineskip}
\noindent Marc Coppens\\
Katholieke Hogeschool Kempen\\
Departement Industrieel Ingenieur en Biotechniek\\
KULeuven Dept. Wiskunde Groep Algebra\\
Kleinhoefstraat 4\\
B 2440 Geel Belgium\\
E-mail: marc.coppens@khk.be\\
Partially supported by the Fund of Scientific Research - Flanders
($G.0318.06$)

\vspace{2\baselineskip}
\noindent Johannes Huisman\\
Universit\'e Europ\'eenne de Bretagne\\
France\\
and\\
Universit\'e de Brest; CNRS, UMR 6205\\
Laboratoire de Math\'ematiques de Brest\\
ISSTB\\
6, avenue Victor Le Gorgeu\\
CS 93837\\
29238 Brest Cedex 3 France\\
E-mail: johannes.huisman@univ-brest.fr\\
Home page: http://pageperso.univ-brest.fr/$\sim$huisman\\

\end{document}